\documentclass[a4paper,english,11pt]{article}
\usepackage[utf8]{inputenc}
\usepackage{babel}
\usepackage{vmargin,pifont}
\usepackage{amsmath}
\usepackage{amsthm}
\usepackage{amsfonts}
\usepackage{amssymb}
\usepackage[arrow, matrix, curve]{xy}
\usepackage{paralist}

\newcommand{\ndash}{\nobreakdash-\hspace{0pt}}

\newtheoremstyle{axel}
{}		
{}		
{\itshape}	
{}		
{\bfseries}	
{}		
{0pt\penalty1000}		
{}		
\newtheoremstyle{axelnl}
{}		
{}		
{\itshape}	
{}		
{\bfseries}	
{}		
{\newline}	
{}		
\theoremstyle{axelnl}
\newtheorem{theorem}{Theorem}[section]
\newtheorem{definition}{Definition}[section]
\newtheorem{cor}{Corollary}[section]
\newtheorem{remark}{Remark}[section]
\newtheorem{lemma}{Lemma}[section]
\newtheorem{example}{Example}[section]

\theoremstyle{axel}
\newtheorem{lemmas}[lemma]{Lemma}
\newtheorem{cors}[cor]{Corollary}

\begin{document}

\title{Kac-Moody groups, analytic regularity conditions and cities}
\author{Walter Freyn}
\date{April 2010}
\maketitle

\begin{abstract}
The relationship between minimal algebraic Kac-Moody groups and twin buildings is well known as is the relationship between formal completions in one direction and affine buildings --- cf.~\cite{Tits84}. Nevertheless, as the completion of a Kac-Moody group in one direction destroys the opposite $BN$\ndash pair, there exists no longer a twin building. For similar reasons, there are so far no buildings at all for analytic completions of affine Kac-Moody groups as studied in \cite{PressleySegal86}. In this article we construct a new type of twin buildings, called twin cities, that are associated to affine analytic Kac-Moody groups over $\mathbb{R}$ and $\mathbb{C}$. Twin cities consist of two sets of buildings. We describe applications of cities in infinite dimensional differential geometry by proving infinite dimensional versions of classical results from finite dimensional differential geometry: For example, we show that points in an isoparametric submanifold in a Hilbert space as described in \cite{Terng89} correspond to all chambers in a city. In a sequel we will describe the theory of twin cities for formal completions.
\end{abstract}

\section{Introduction}

Affine Kac-Moody algebras can be viewed as $2$\ndash dimensional extensions of (twisted) polynomial loop algebras. Completing those polynomial algebras with respect to various norms, we can construct various completions of Kac-Moody algebras, denoted $\widehat{L}(\mathfrak{g}, \sigma)$. The resulting algebras are called Kac-Moody algebras of the ``analytic'' level, in contrast for example to formal completions --- cf.~\cite{Tits84}. Depending on the norms, those algebras are Hilbert\ndash, Banach\ndash, Fréchet\ndash, etc.\ Lie algebras. 

In this article we develop the theory of twin buildings associated to affine Kac-Moody algebras and Kac-Moody groups of the ``analytic'' level.

The construction of affine twin buildings for ``minimal'' affine Kac-Moody groups is well known since 25 years --- cf.~\cite{AbramenkoBrown08}. Nevertheless as the completion of a Kac-Moody group destroys the twin $BN$\ndash pair of the Kac-Moody group, this theory fails for all completions. As analytic completions of Kac-Moody groups and Kac-Moody algebras play an important role in infinite dimensional differential geometry --- cf.~\cite{PressleySegal86}, \cite{Terng95}, \cite{Heintze06}, \cite{Freyn09}, and \cite{Freyn10a} --- this is a serious detriment.

We solve this problem by the introduction of geometric twin $BN$\ndash pairs and their associated cities\footnote{In \cite{Freyn09} and \cite{Freyn10a} we called cities ``universal geometric twin buildings'', but changed the denomination, being adverted of the risk of confusion with the universal ``algebraic'' twin building, introduced by Ronan and Tits --- cf.~\cite{RonanTits94}.}. Those cities consist of two parts, denoted $\mathfrak{B}^+$ and $\mathfrak{B}^-$, that each consist of an infinity of affine buildings. Furthermore each pair $(\Delta^+, \Delta^-)$, consisting of an affine building $\Delta^+$ in $\mathfrak{B}^+$ and an affine building $\Delta^-$ in $\mathfrak{B}^-$, is an algebraic affine twin building. The main results of this article can be summarized as follows:

\begin{theorem}[Twin cities]
   For each analytic Kac-Moody group $\mathcal{G}$ there is an associated twin city $\mathfrak{B}= \mathfrak{B}^+ \cup \mathfrak{B^-}$, such that
   \begin{enumerate}[(i)]
     \item Each connected component $\Delta^{\pm}$ in $\mathfrak{B}^{\pm}$ is an affine building.
     \item Each pair $(\Delta^+, \Delta^-)\in \mathfrak{B}^+\cup
      \mathfrak{B}^-$, consisting of a building in $\mathfrak{B}^+$ (``positive'' building) and one in $\mathfrak{B}^-$ (``negative'' building),
     is an affine twin building.
     \item $\mathfrak{B}=\mathfrak{B}^+\cup \mathfrak{B}^-$ has a spherical building at infinity.
     \item $\mathcal{G}$ acts on its twin city.
     \item ``Small'' twin cities, associated to Kac-Moody groups, defined by stronger regularity conditions, embed into ``big'' twin cities, associated to Kac-Moody groups, defined by weaker regularity conditions.
   \end{enumerate}
\end{theorem}

Recall, that an affine Kac-Moody algebra $\widehat{L}(G, \sigma)$ admits an invariant bilinear form of Lorentz signature and denote by $c$ and $d$ the generators of the $2$\ndash dimensional extension of the loop algebra. 
The relationship between Kac-Moody algebras and twin cities is described in the following result:

\begin{theorem} [Embedding of cities]
Denote by $H_{l, r}$ the intersection of the sphere of radius $l\in \mathbb{R}$  of a real affine Kac-Moody algebra $\widehat{L}(\mathfrak{g}, \sigma)$ with the horospheres $r_d=\pm r\not=0$, where $r_d$ is the coefficient of $d$ in $\widehat{L}(\mathfrak{g}, \sigma)$.
  There is a $2$\ndash parameter family $\varphi_{l,r}, (l,r)\in \mathbb{R}\times \mathbb{R}^+$ of $\widehat{L}(G, \sigma)$\ndash equivariant immersions of the  twin city $\mathfrak{B^+\cup \mathfrak{B^-}}$ into  $\widehat{L}(\mathfrak{g},\sigma)$. It  identifies $\mathfrak{B}$ with $H_{l, r}$. The two parts of the city
 $\mathfrak{B}^+$ and $\mathfrak{B}^-$ are immersed into the two sheets of $H_{l,r}$ described by $r_d<0$ resp.\ $r_d>0$ of the space $H_{l, r}$.
\end{theorem}

We describe the content of this article in more detail:

Section~\ref{section:summary_of_the_algebraic_theory} is devoted to a short summary of the theory of affine twin buildings.

In section~\ref{section:regularity} we introduce the most important regularity conditions and study the embeddings of smaller Kac-Moody groups (defined using stronger regularity conditions) into bigger ones (defined using weaker regularity conditions). We call a subgroup of $L(G, \sigma)$, which is isomorphic to
$L_{\textrm{alg}}G^{\sigma}$, a quasi-algebraic group. The main result of section~\ref{section:regularity} is that all quasi-algebraic subgroups are conjugate.
Furthermore we study the relationship between Borel subgroups in $L_{\textrm{alg}}G^{\sigma}$ and Borel subgroups in $L(G, \sigma)$ (resp.\ $\widehat{L}(G, \sigma)$).

Section~\ref{section:twin_cities} contains the core of the article: we define geometric $BN$\ndash pairs and their cities, which are the generalizations of $BN$\ndash pairs and buildings to completed groups, and study group actions on them.

In section~\ref{section:twin_cities_and_kac_moody_algebras} we describe the connection between polar actions and cities. In particular we prove that points in isoparametric PF submanifolds in a Hilbert space, as introduced in \cite{Terng89}, correspond to chambers in the twin city of $H^1$\ndash regularity.

In section~\ref{section:topology_and_geometry} we investigate the topology of the space of chambers, the space of affine buildings and the twin city. We define a ultrametric pseudo distance on the space of affine buildings.

In section~\ref{section:the_spherical_building_at_infinity} we study the spherical building at infinity. For each affine twin building, we get a spherical building over the field of rational functions. Furthermore --- interpreting the loop groups as realizations of an affine algebraic group scheme $G$ over a ring of functions $R$ --- we study the spherical building over the quotient field $Q(R)$. This gives a spherical building for $\mathfrak{B}$, i.e.\ a spherical building $\mathfrak{B}^{\infty}$ such that the spherical buildings $\Delta^{\infty}$ of $\Delta^{\pm}$ embed into $\mathfrak{B}^{\infty}$ as subbuildings.

In a sequel to this paper, we will describe an abstract theory of cities and investigate the relationship with universal algebraic twin buildings, introduced by Marc Ronan and Jacques Tits in \cite{RonanTits94} and \cite{RonanTits99}.

\section{Summary of the algebraic theory}
\label{section:summary_of_the_algebraic_theory}

In this section we collect basic results about algebraic affine twin buildings for Kac-Moody groups. Two references for this section are \cite{AbramenkoBrown08} and \cite{Remy02}.

\begin{samepage}
\begin{definition}[Twin $BN$\ndash pair] 
Let $\widehat{L_{\textrm{alg}}G_{\mathbb{C}}}$ be a complex Kac-Moody group. $(B^+, B^-, N, W, S)$  is a twin $BN$\ndash pair for $\widehat{L_{\textrm{alg}}G_{\mathbb{C}}}$ iff:
\begin{enumerate}
	\item $(B^+, N, W, S)$ is a $BN$\ndash pair  (called $B^+N$),	
	\item $(B^-, N, W, S)$ is a $BN$\ndash pair  (called $B^-N$),
	\item $B^+N$ and $B^-N$ are compatible, i.e.
    \begin{enumerate}
    \item If $l(ws)<l(w)$ then $B^{\epsilon}wB^{-\epsilon}sB^{-\epsilon}= B^{\epsilon}wsB^{-\epsilon}$ for $\epsilon \in \{+,-\}$, $w\in W, s \in S$,
    \item $B^+s\cap B^-=\emptyset \hspace{2pt}\forall s \in S$.    
    \end{enumerate}
\end{enumerate}  
\end{definition}
\end{samepage}

We use the notation $w(f)$ (resp.\ if there is danger of ambiguity: $w^{\epsilon}(f)$, $w^{\pm}(f)$, $w^{\mp}(f)$) to denote the class of $f$ in the corresponding Bruhat decomposition.
The existence of the twin $BN$\ndash pairs yields Bruhat decompositions for $\widehat{L_{\textrm{alg}}G_{\mathbb{C}}}$ similar to the finite dimensional case:

\begin{theorem}[Bruhat decomposition]
Let $\widehat{L_{alg}G}$ be an affine algebraic Kac-Moody group with affine Weyl group $W_{\textrm{aff}}$.
Let furthermore $B^\pm$ denote a positive (resp.\ negative) Borel group. There are decompositions
$$\widehat{L_{alg}G}=\coprod_{w\in W_{\textrm{aff}}}B^+wB^+=\coprod_{w\in
W_{\textrm{aff}}}B^-wB^-\,.$$
\end{theorem}

\begin{theorem}[Bruhat twin decomposition]
Let $\widehat{L_{alg}G}$ be an affine algebraic Kac-Moody group with affine Weyl group $W_{\textrm{aff}}$.
Let furthermore $B^\pm$ denote a positive and its opposite negative Borel group. There are decompositions
 $$\widehat{L_{alg}G}=\coprod_{w\in W_{\textrm{aff}}}B^{\epsilon}wB^{-\epsilon}\hspace{3pt} \epsilon \in \{+,-\}\,.$$ 
\end{theorem} 

For proofs see any book about Kac-Moody groups, i.e.~\cite{Remy02}, chapter~1. What we call Bruhat twin decomposition is sometimes called Birkhoff decomposition.

Note that the Bruhat decompositions and the Bruhat twin decompositions are defined on the whole group $\widehat{L_{alg}G}$. For the associated buildings, this translates into the fact that any two chambers in $\mathfrak{B}^+$ resp.\ $\mathfrak{B}^-$ have a well-defined Weyl distance and a well-defined Weyl codistance (compare definition~\ref{algebraictwinbuildingwmetricapproach}).

The same results hold for $L_{alg}G$.

\begin{samepage}
\begin{definition}[$BN$\ndash flip]
An involution $\varphi$ of a Kac-Moody group is called a $BN$\ndash flip iff
\begin{enumerate}
	\item  $\varphi^2=1$,
	\item $\varphi(B^+)=B^-$,
	\item $\varphi$ centralizes the Weyl group.
\end{enumerate}
\end{definition}
\end{samepage}

A $BN$\ndash flip swaps the two $BN$\ndash pairs. The existence of a $BN$\ndash flip is a sign of symmetry of the group structure, which can be lost for non\ndash algebraic Kac-Moody groups.

Similarly to the two conjugacy classes of Borel subgroups, the set of affine parabolic subgroups breaks up into two classes. The first one consists of affine parabolic subgroups containing a conjugate of $B^+$, the second one of those containing a conjugate of $B^-$.  
The two sets of parabolic subgroups admit a partial order relation exactly as in the finite dimensional case. To this complex, one can associate a simplicial complex, which has the structure of an affine Tits building.  The apartments are Coxeter complexes for $W_{\textrm{aff}}$.

A big difference between affine buildings and spherical ones is that chambers in affine buildings do not have opposite chambers
(recall, that two chambers are called opposite if their Weyl distance is maximal, which is not possible in an affine Weyl group). As the existence of opposite chambers is a necessary ingredient for various structure results, this is a serious detriment. The most important consequence of the existence of opposite chambers for spherical buildings is the following theorem:

\begin{theorem}\label{apartmentsconvexhulls}
In a spherical building, apartments are exactly the convex hulls of a pair of opposite chambers.
\end{theorem}

\begin{proof} cf.~\cite{AbramenkoBrown08}. \end{proof}

\noindent This theorem implies the corollary:

\begin{cor}
\label{cor:apartmentsystem_unique}
The apartment system in a spherical building is unique.
\end{cor}

In contrast affine buildings have various different apartment systems, reflecting the different completions of the associated Kac-Moody groups. We will meet some of those apartment systems in section~\ref{section:the_spherical_building_at_infinity}.

So there is a need for a version of the concept of opposite chambers for affine buildings, which should then lead to a generalization of theorem~\ref{apartmentsconvexhulls}. It is clear that an opposite chamber cannot be in the same building, as it would have a Weyl distance of maximal length.  Thus the solution lies in a twinning of the two buildings associated to the two $BN$\ndash pairs.  The resulting object, called a twin building, behaves in many respects like a spherical building.

\noindent We quote the $W$\ndash metric definition of a building from the monograph~\cite{AbramenkoBrown08}, chapter~5.  
\ndash
\begin{definition}[building]
\label{wmetricbuilding}
A building of type $(W,S)$ is  a pair $(\mathcal{C}, \delta)$  consisting of a nonempty set $\mathcal{C}$ whose elements are called chambers together with a map $\delta: \mathcal{C}\times \mathcal{C}\longrightarrow W$, called the Weyl distance function, such that for all $C,D\in \mathcal{C}$ the following conditions hold: 
\begin{enumerate}
	\item $\delta(C, D)=1$ iff $C=D$.
	\item If $\delta(C,D)= w $ and $C'\in \mathcal{C}$ satisfies $\delta(C', C)=s\in S$ then $\delta(C',D)=sw$ or $w$. If, in addition, $l(sw)=l(w)+1$ then $\delta(C', D)=sw$.
	\item If $\delta(C,D)=w$ then for any $s\in S$ there is a chamber $C' \subset \mathcal{C}$, such that $\delta(C', C)=s$ and $\delta(C', D)=sw$. 
\end{enumerate}
\end{definition}

This definition coincides with the classical definition of a building as a simplicial complex. For a proof cf.~\cite{AbramenkoBrown08}. The construction of the apartments is somewhat involved. 

\begin{samepage}
\begin{definition}[Twin building]
\label{algebraictwinbuildingwmetricapproach} 
A twin building of type $(W,S)$ is a quintuple $(\mathcal{C}^+, \mathcal{C}^-, \delta^+, \delta^-, \delta^*)$ such that
\begin{enumerate}
	\item $(\mathcal{C}^+, \delta^+)$ is a building of type $(W,S)$,
	\item $(\mathcal{C}^-, \delta^-)$ is a building of type $(W,S)$,
	\item $\delta^*$ is a codistance  i.e.\ for $X \in \mathfrak{B}^{\epsilon}$ and $Y, Z \in \mathfrak{B}^{-\epsilon}$,
		\begin{enumerate}
		\item[a)] $\delta^*(X,Y)= \delta^*(Y,X)^{-1}$,
		\item[b)] $\delta^*(X,Y)=w, \delta(Y,Z)=s \in S$ and $l(ws)=l(w)-1$, then $\delta^*(X,Z)=ws$,
		\item[c)] $\delta^*(X,Y)=w$ and  $s \in S$. Then there is $Z\in \mathfrak{B}^-$ such that $\delta^*(Y,Z)=s$ and $\delta^*(X,Z)=ws$.
		\end{enumerate}
\end{enumerate}
\end{definition}
\end{samepage}

\begin{definition}
$X$ and $Y$ are called opposite iff $\delta^*(X,Y)=1$.
\end{definition}

For a pair of affine buildings the twinning is in general not uniquely determined. There are uncountable many non-isomorphic twinnings. In the case of rank $1$\ndash buildings -- that is trees --, a universal twin building has been constructed by Mark Ronan and Jacques Tits --- cf.~\cite{RonanTits94}, \cite{RonanTits99}. For more general classes of buildings this is an open problem. 

A twinning can be described via twin apartments. Here, we have again the result that the system of twin apartments is well defined. Furthermore, twin apartments are the convex hulls of opposite chambers --- cf.~\cite{AbramenkoRonan98}.

\section{Regularity and Kac-Moody theory}
\label{section:regularity}

To an affine Cartan matrix  and a field $\mathbb{F}\in \{\mathbb{R}, \mathbb{C}\}$ there is attached a great variety of different infinite dimensional Lie algebras: The minimal one is the algebraic Kac-Moody algebra, corresponding to a Lie algebra of polynomial maps. All further Kac-Moody algebras arise as completions: On the one hand, understanding this algebra as an extension of a Lie algebra over the abstract polynomial ring $\mathbb{F}[t,t^{-1}]$, we can turn to the formal completion and study the formal Kac-Moody algebras. On the other hand, taking the point of view of polynomial maps on $S^1$ resp.\ $\mathbb{C}^*$, we get a great variety of \glqq analytic\grqq\/ completions and Kac-Moody algebras associated to them --- cf.~\cite{Tits84}.

Let $\mathfrak{g}$ be a simple complex or compact Lie algebra and $\sigma$ a diagram automorphism. Denote the associated loop algebra by $L(\mathfrak{g}, \sigma):=\{f: \mathbb{R}\longrightarrow \mathfrak{g}| f(t+2\pi)=\sigma f(t), f$ satisfies some regularity condition$\}$.
The associated Kac-Moody algebra is defined by $\widehat{L}(\mathfrak{g}, \sigma)=L(\mathfrak{g}, \sigma)\oplus \mathbb{F}c\oplus \mathbb{F}d$ where $d$ acts on $L(\mathfrak{g}, \sigma)$ as a derivative and $c$ is a central element. Hence $[d,f]=f'$, $[c,d]=[c,f]=0$ and $[f,g]=[f,g]_0+\omega(f,g)c$, where $[f,g]_0$ denotes the Lie bracket of $L(G, \sigma)$ and $\omega$ is a certain antisymmetric $2$\ndash form --- cf.~\cite{Kac90}.

Similarly, to a given affine root datum, we can attach a variety of different Kac-Moody groups. In his overview  \cite{Tits84}, Jacques Tits describes realizations of a root datum at the algebraic, the formal and the analytic level. In this article, aside from the algebraic Kac-Moody groups, we will encounter various Kac-Moody groups of the analytic level. These groups are closely related: The algebraic one is contained in all other Kac-Moody groups. At the analytic level, Kac-Moody groups defined using stronger regularity conditions embed naturally into Kac-Moody groups defined using weaker regularity conditions. Nevertheless, all those embeddings are by no means unique.

Similarly to affine Kac-Moody algebras that can be viewed as $2$\ndash dimensional extensions of certain loop algebras, affine Kac-Moody groups can be realized as torus extensions $\widehat{L}(G, \sigma)$ of groups of maps $$L(G, \sigma)=\left\{f:\mathbb{R}\longrightarrow G\mid f(t+2\pi)=\sigma f(t), f \textrm{ satisfies some regularity condition}\right\}\,,$$ where $G$ denotes a simple complex or compact Lie group and $\sigma$ a diagram automorphism of $G$. We construct first $\widetilde{L}(G, \sigma)$ as a central $S^1$\ndash\ resp.\ $\mathbb{C}^*$\ndash extension, corresponding to the $c$\ndash term of the Kac-Moody algebra, then we take a semidirect product with $S^1$ resp.\ $\mathbb{C}^*$ whose action on $L(G, \sigma)$ is defined to be a shift of the argument: $w\cdot f(t)=f(t+w)$. For this shift to be well defined we need in the complex setting a continuation of $f$ to $\mathbb{C}$ --- cf.~\cite{PressleySegal86}, \cite{Popescu05}, \cite{Freyn09}, \cite{Khesin09}, and various other sources.

 Depending on their loop realizations, affine Kac-Moody algebras (resp.\ groups) break up into two subclasses: the twisted ones, i.e.\ those with $\sigma \neq\textrm{Id}$,  and the untwisted ones (i.e.\ those with $\sigma=\textrm{Id}$). Twisted affine Kac-Moody algebras (resp.\ groups) can be described as subsets of untwisted ones. Hence, it is sufficient to describe the regularity conditions for untwisted Kac-Moody algebras (resp.\ groups).

Let us mention several widely used regularity conditions.
For the description of most of them we need only the analytic structure of the Lie group $G$. Examples are the groups of continuous loops $LG$, $k$\ndash differentiable loops $L^kG^{\sigma}$, smooth loops $L^{\infty}G^{\sigma}$, real analytic or complex analytic loops $MG^{\sigma}$ resp.\ $A_nG^{\sigma}$ on $X\in \{\mathbb{C}^*, A_n:=\{z\in \mathbb{C}|e^{-n}\leq |z|\leq e^n\}\}$ (for the last two cases to make sense, we need $G$ to be a complex Lie group).

In contrast the precise meaning of algebraic (or polynomial) loops into a Lie group is not clear a priori.

The algebraic loop (resp.\ Kac-Moody) group is the smallest group, we are interested in: We put

$$L_{\textrm{alg}}G^{\sigma}:=\{f \in L(G, \sigma)|\textrm{ has a finite Fourier expansion}\}.$$
where the Fourier expansion is defined via the adjoint representation of $G$. If $G$ is complex, we can identify this group with a group of matrix\ndash valued Laurent polynomials --- cf.~\cite{PressleySegal86}. By construction, this group is isomorphic to the group $G(\mathbb{C}[t,t^{-1}])$, the realization of the affine algebraic group scheme corresponding to the Lie group $G$ over the ring $\mathbb{C}[t,t^{-1}]$ (for the definition, cf.~\cite{AbramenkoBrown08} or \cite{Waterhouse79}. $G(\mathbb{C}[t,t^{-1}])$ is the group acting in a natural way on a twin building --- cf.~\cite{Ronan03}, \cite{Remy02}.

Let us now investigate some relationships between these regularity conditions.

If $G$ is semisimple, then $L_{\textrm{alg}}G$ is dense in the group of continuous loops $LG$ --- cf.~\cite{PressleySegal86}, chapter~3.5. Hence $\widehat{L_{\textrm{alg}}G}^{\sigma}$ is dense in $\widehat{L}(G ,\sigma)$. Thus we can interpret the Kac-Moody groups $\widehat{L}(G, \sigma)$ for $G$ semisimple as completions
of the algebraic Kac-Moody groups $\widehat{L_{alg}G}^{\sigma}$. Then $\widehat{L_{alg}G}^{\sigma}$ is in a natural way
a subgroup of $\widehat{L}(G, \sigma)$.
Call a subgroup of $\widehat{L}(G, \sigma)$ which is isomorphic to the group of algebraic loops $L_{\textrm{alg}}G^{\sigma}$ a quasi\ndash algebraic subgroup --- denoted $\widehat{L_{\textrm{qalg}}G}^{\sigma}$ --- of $\widehat{L}(G, \sigma)$.

$\widehat{L_{\textrm{alg}}G}^{\sigma}$ is clearly not the only quasi\ndash algebraic subgroup of $\widehat{L}(G, \sigma)$. For example for any $f\in \widehat{L}(G, \sigma)$ the subgroup 
$ \widehat{L_{alg}G}_f^{\sigma}:=  f\widehat{L_{alg}G}^{\sigma} f^{-1}$ is quasi\ndash algebraic.

We want to show that all quasi\ndash algebraic groups are conjugate. To this end, we need the following definition:

\begin{definition}
A group $\widehat{L}(G, \sigma)$ satisfies the conjugation property iff all tori of finite type are conjugate.
\end{definition}

The conjugation property is satisfied for example by the groups $\widehat{MG}_{\mathbb{F}}^{\sigma}$, $\widehat{A_nG}^{\sigma}$ by the following theorem:

\begin{theorem}
\label{polaractiononmg}
All tori of finite type in $\widehat{XG}_{\mathbb{F}}^{\sigma}$, $X\in \{A_n , M\}$ are conjugate.
\end{theorem}

\begin{proof} cf.~\cite{Freyn09}. \end{proof}

Furthermore, the conjugation property is satisfied by the groups of $H^1$\ndash Sobolev loops --- cf.~\cite{Terng95}, of $k$\ndash times differentiable loops --- cf.~\cite{Freyn09}, and of smooth loops --- cf.~\cite{Popescu05}, \cite{Freyn09}.

\begin{theorem}[Algebraic subgroups]
\label{algebraic_subgroups}
Let $G$ be a simple, simply connected compact or complex Lie group. Suppose $\widehat{L}(G, \sigma)$ satisfies the conjugation property. Let $\widehat{L_{qalg}G}^{\sigma}\subset \widehat{L}(G, \sigma)$ be a quasi\ndash algebraic subgroup. Then there is $f\in \widehat{L}(G, \sigma)$ such that $\widehat{L_{qalg}G}^{\sigma}:= f \widehat{L_{\textrm{alg}}G}^{\sigma} f^{-1}$.
\end{theorem}

\noindent For the proof we use the following observation --- cf.~\cite{PressleySegal86}, proposition~5.2.5:

\begin{lemma}
\label{generation_of_l_alg_G}
If $G$ is a simple, simply connected compact Lie group of rank $l$, then the $l+1$ subgroups $i_{\alpha}(SU(2))$ corresponding to the simple affine roots generate $\widetilde{L_{alg}G}$.
\end{lemma}

If $G$ is not simply connected, then $LG$ is not connected. In this case the subgroups $i_{\alpha}(SU(2))$ generate the identity component $(\widetilde{L_{alg}G})_0$. The proof of \cite{PressleySegal86} generalizes to the case of twisted loop groups, as it relies only on the algebraic structure of generators and relators of the Kac-Moody algebra.

A similar result holds for complex Lie groups:

\begin{lemma}
\label{generation_of_l_alg_G_C}
If $G_{\mathbb{C}}$ is a simple, simply connected complex Lie group of rank $l$, then the $l+1$ subgroups $i_{\alpha}(SL_2(\mathbb{C}))$ corresponding to the simple affine roots generate $\widetilde{L_{alg}G_{\mathbb{C}}}$.
\end{lemma}

This is the loop group version of the description of complex affine Kac-Moody groups as the amalgam of its $SL_2(\mathbb{C})$ subgroups --- cf.~\cite{Caprace09}, \cite{Remy02}.

Hence a quasi\ndash algebraic subgroup $\widehat{L_{qalg}G}^{\sigma}$ of a group $\widehat{L}(G, \sigma)$ is completely described by the choice of a maximal torus. An embedding $\varphi:\widehat{L_{alg}G}^{\sigma}\longrightarrow \widehat{L}(G, \sigma)$ is determined by the choice of an isomorphism of a torus and the family of isomorphisms of the $SU(2)$\ndash\ resp.\ $SL_2(\mathbb{C})$\ndash subgroups corresponding to the simple roots. Hence we get the following lemma:

 \begin{lemma}
  Let $G$ be a simple, simply connected compact or complex Lie group and let $\varphi_i:\widehat{L_{\textrm{alg}}G}^{\sigma}\longrightarrow \widehat{L}(G, \sigma), i=\{1,2\}$ be two embeddings, let $\widehat{T}$ be a maximal torus in $\widehat{L_{\textrm{alg}}G}^{\sigma}$ and $i_{\alpha}(SU(2))$ the subgroups associated to the simple roots.

If $\varphi_1(\widehat{T})=\varphi_2(\widehat{T})$ and  $\varphi_1(i_{\alpha}SU(2))=\varphi_2(i_{\alpha}SU(2)) \forall \alpha$, then $\varphi_1=\varphi_2$.
 \end{lemma}

\begin{proof}
Two group isomorphism are equivalent iff they coincide on a set of generators. By lemma~\ref{generation_of_l_alg_G} (for compact $G$) resp.\ \ref{generation_of_l_alg_G_C} (for complex $G$), $\widehat{L_{alg}G}^{\sigma}$ is generated by a torus and the $SU(2)$\ndash\ resp.\ $SL_2( \mathbb{C})$\ndash subgroups associated to the simple roots. This proves the lemma.
\end{proof}

We are now in a position to prove theorem~\ref{algebraic_subgroups}:
\begin{proof}
 Let $G_i \subset \widehat{L}(G, \sigma), i \in \{1,2\}$ be two quasi\ndash algebraic subgroups of $\widehat{L}(G, \sigma)$. Choose two tori of finite type $\widehat{T}_1\subset G_1$ and $\widehat{T}_2\subset G_2$. By the conjugation property, there is $g\in \widehat{L}(G, \sigma)$ such that $\widehat{T}_1= g \widehat{T}_2 g^{-1}$.
Define $H_2:= g G_2 g^{-1}$. The groups $G_1$ and $H_2$ share the maximal torus $\widehat{T}_1$. Hence they have the same root system with respect to $\widehat{T}_1$. Choose in both groups the same system of simple roots. The $SU(2)$ resp.\ $SL_2(\mathbb{C})$\ndash subgroups associated to those simple roots are well defined. Hence they coincide. Thus by lemma~\ref{generation_of_l_alg_G}, $G_1$ and $H_2$ coincide. Thus $G_1$ and $G_2$ are conjugate.
\end{proof}

Next we investigate Borel subgroups and parabolic subgroups.

\begin{definition}[countably solvable subgroup]
 A subgroup $S\subset \widehat{L}(G, \sigma)$ is called countably solvable if the upper central series converges to the identity.
\end{definition}

\begin{definition}
A Borel group of $\widehat{L}(G,\sigma)$ is a maximal countably solvable subgroup of $\widehat{L}(G,\sigma)$.
\end{definition}

Other definitions are proposed in \cite{Remy02} and \cite{Kumar02}.

As usual, two Borel subgroups $B^+$ and $B^-$ in a Kac-Moody group $\widehat{L}(G, \sigma)$ are opposite, iff $B^+ \cap B^-\simeq \widehat{T}$. Call $\widehat{N}$ the normalizor of $\widehat{T}$ and put $W=\widehat{N}/\widehat{T}$. $\widehat{N}$ and hence also $W$ are independent of the regularity of $\widehat{L}(G, \sigma)$.

Let $B^+$ and $B^-$ be two opposite Borel subgroups. Then they describe a unique quasi-algebraic subgroup $G(B^+, B^-) \subset \widehat{L}(G, \sigma)$: We construct this group by taking the torus $\widehat{T}=B^+\cap B^-$ and taking the group generated by $\widehat{T}$ and the simple root groups. This group has a twin $BN$\ndash pair consisting of $(B^+\cap G(B^+, B^-),B^-\cap G(B^+, B^-), N, W, S)$.

Hence, given a pair of opposite Borel subgroups, we get in each of the two Borel subgroup a Borel subgroup corresponding to the quasi-algebraic group $G(B^+, B^-)$. We call this Borel subgroup the algebraic kernel of $B^{\epsilon}$ ($\epsilon \in \pm$) with respect to $B^{-\epsilon}$.

More generally, let $B$ be a Borel group. As $B$ is a countably solvable group, we get a series of subgroups $B^{(i+1)}=[B^{(i)},B^{(i)}]$. Those groups are well defined. But we don't get complementary subspaces. Those have to be chosen explicitly. We denote by $T^{(i)}$ a series of subspaces such that $B^{(i+1)}\simeq T^{(i)}\oplus B^{(i)}$.

$T^{(0)}$ is just a maximal torus of finite type. $T^{(0)}$ uniquely determines a set of root subgroups. Define $T^{(1)}$ to be the union of the root subgroups of simple roots and similarly for $T^{(i)}:=\{ab|a,b\in T^{(i-1)}\}\cap B^{(i+1)}, i>0$. From now on, by $T^{(i)}$ we will denote this set of subspaces.

\begin{definition}[$T$\ndash algebraic subgroup]
 The $T$\ndash algebraic subgroup $B_{alg}^T$ of a Borel subgroup $B$ with respect to a maximal Torus $T$ is the subgroup of elements $f\subset B$ such that there exist some $n\in \mathbb{N}$ such that $f\subset \displaystyle\bigcup_{i=0}^n T^{(i)}$.
\end{definition}

\begin{lemma}
 Let $G(B^+, B^-)$ be a quasi-algebraic subgroup of $\widehat{L}(G, \sigma)$ and $T$ a torus in $G(B^+, B^-)$  that is a complement to $B^{+,(1)}$ in $B^+$. Then
$$B^+\cap G(B^+, B^-)= B_{alg}^{+,T}\,.$$
\end{lemma}

\begin{proof}
 Let $T$ be a torus in $G(B^+, B^-)$ as described in the lemma. We claim: There is a Borel subgroup $\widetilde{B}^{-}$ such that $T=B^+\cap \widetilde{B}^-$ and $G(B^+,  \widetilde{B}^-)=G(B^+, B^{-})$. Then the lemma follows from the definitions. 
Hence we are left with proving our claim: We construct $\widetilde{B}^-$ explicitly to be the completion of  the negative root subgroups.
\end{proof}

\begin{lemma}
All positive (resp.\ negative) Borel subgroups $B^+$ (resp.\ $B^-$) of $\widehat{L}(G, \sigma)$ are conjugate.
\end{lemma}

In the case $GL(n,\mathbb{C)}$, the proof proceeds as follows: Each Borel subgroup fixes a unique maximal flag. Hence the result follows as $GL(n, \mathbb{C})$ is transitive on maximal flags. If $G\subset GL(n,\mathbb{C})$ the proof follows using the embedding.

\begin{proof}
The proof studies the action of $L(G, \sigma)$ on a suitable vector space and shows, that it is transitive on spaces of periodic flags.
For details cf.~\cite{PressleySegal86}, section~7 and 8 and \cite{Freyn10b}.
\end{proof}

We summarize our results:
Each pair of two opposite Borel subgroups $B^{\pm}$ defines exactly one maximal torus and hence a quasi-algebraic subgroup of $\widehat{L}(G, \sigma)$. We denote this group $G(B^+, B^-)$. Quasi algebraic subgroups, such that a given Borel group $B$ is the completion of an algebraic Borel subgroup $B_{alg}$, contain a torus which is a complement to $B^{(1)}$ in $B$.

\section{Twin cities}
\label{section:twin_cities}

In this section we construct cities associated to simple geometric affine Kac-Moody algebras $\widehat{L}(\mathfrak{g}, \sigma)$ and their Kac-Moody groups $\widehat{L}(G, \sigma)$.

There are two major obstacles:
\begin{enumerate}
\item Twin buildings correspond only to the subgroup of algebraic loops, or taking into account that the subgroup of algebraic loops is just one distinguished element in the conjugacy class of quasi-algebraic groups, quasi-algebraic groups.
\item The completions of Kac-Moody groups do not act properly on twin buildings.
\end{enumerate}

To resolve those problems, we define geometric $BN$\ndash pairs and their associated cities. Those cities are chamber complexes such that $\mathfrak{B}^+$ and $\mathfrak{B}^-$ each consist of an infinite number of connected components, each of which is an affine building, such that each pair consisting of a building $\Delta^+$ in $\mathfrak{B}^+$ and a building $\Delta^-$ in $\mathfrak{B}^-$ is a twin building in the classical algebraic sense.

\subsection{Geometric \boldmath $BN$\ndash pairs}

\begin{definition}[Geometric $BN$\ndash pair for $\widehat{L}(G,\sigma)$]

Let $\widehat{L}(G,\sigma)$ be an affine Kac-Moody group. 
$(B^+, B^-, N, W, S)$  is a twin $BN$\ndash pair for $\widehat{L}(G,\sigma)$ iff there are subgroups $\widehat{L}(G, \sigma)^+$ and $\widehat{L}(G, \sigma)^-$ of $\widehat{L}(G,\sigma)$ such that $\widehat{L}(G,\sigma)=\langle \widehat{L}(G, \sigma)^+, \widehat{L}(G, \sigma)^- \rangle$ is subject to the following axioms:
\begin{enumerate}
	\item $(B^+, N, W, S)$ is a $BN$\ndash pair for $\widehat{L}(G,\sigma)^+$ (called $B^+N$),	
	\item $(B^-, N, W, S)$ is a $BN$\ndash pair for $\widehat{L}(G,\sigma)^-$ (called $B^-N$),
	\item $(B^+\cap \widehat{L}(G, \sigma)^- , B^-\cap \widehat{L}(G, \sigma)^+, N, W, S)$ is a twin $BN$\ndash pair for $\widehat{L}(G, \sigma)^+\cap \widehat{L}(G, \sigma)^-$.
\end{enumerate}
\end{definition}

The subgroups $\widehat{L}(G, \sigma)^+$ and $\widehat{L}(G, \sigma)^-$ of $\widehat{L}(G,\sigma)$ depend on the choice of $B^+$ and $B^-$. A choice of a different subgroup $B^{+'}$ (resp.\  $B^{-'}$) gives the same  subgroup $\widehat{L}(G, \sigma)^+$ (resp.\ $\widehat{L}(G, \sigma)^-$) of $\widehat{L}(G,\sigma)$ if $B^{+'}\subset \widehat{L}(G, \sigma)^+$ (resp.\ $B^{-'}\subset \widehat{L}(G,\sigma)$ if $B^{-'}\subset \widehat{L}(G, \sigma)^-)$. For all positive (resp.\ negative) Borel subgroups the positive (resp.\ negative) subgroups $\widehat{L}(G, \sigma)^+$ resp.\ $\widehat{L}(G, \sigma)^-$ are conjugate. Hence without loss of generality we can think of $B^{\pm}$ to be the standard positive (resp.\ negative) affine Borel subgroup. The groups $\widehat{L}(G, \sigma)^{\pm}$ --- called the standard positive (resp.\ negative) subgroups --- are then characterized by the condition that $0$ (resp.\ $\infty$) is of finite order for all elements.

\begin{remark}
For an algebraic Kac-Moody group a geometric $BN$\ndash pair coincides with a $BN$\ndash pair. Hence we get $\widehat{L}(G,\sigma)^+ =\widehat{L}(G,\sigma)^-=\widehat{L}(G,\sigma)$.
\end{remark}

We use the equivalent definitions for the loop groups $L(G, \sigma)$.

The $W$\ndash metric description of buildings shows that the structure of a twin building is intimately related to the Bruhat and the Bruhat twin decomposition. For completed Kac-Moody groups, those decompositions need no longer be globally defined. This new feature is crucial for the disconnected structure of cities.

\begin{lemmas}~
\label{bruhatforloopgroups}
\begin{enumerate}
 	\item The groups $L(G,\sigma)^+$ (resp.\ $L(G,\sigma)^-$) have a positive (resp.\ negative) Bruhat decomposition and a Bruhat twin decomposition.
	\item The group $L(G,\sigma)$ has a Bruhat twin decomposition but no Bruhat decomposition.
\end{enumerate}
\end{lemmas}

\begin{proof}
The Bruhat decomposition in the first part follows by definition, the Bruhat twin decomposition by restriction and the second part.
The second part is a restatement of the decomposition results in chapter~8  of~\cite{PressleySegal86}. 
\end{proof}

Compare also similar decomposition results stated in \cite{Tits84}.

\begin{theorem}[Bruhat decomposition]
Let $\widehat{L}(G, \sigma)$ be an affine Kac-Moody group with affine Weyl group $W_{\textrm{aff}}$.
Let furthermore $B^\pm$ denote a positive (resp.\ negative) Borel group. There are decompositions
\begin{align*}
\widehat{L}(G, \sigma)^+&=\coprod_{w\in W_{\textrm{aff}}}B^+wB^+\vspace{-3pt}
\intertext{and}\vspace{-3pt}
\widehat{L}(G, \sigma)^-&=\coprod_{w\in W_{\textrm{aff}}}B^-wB^-\,.
\end{align*}
\end{theorem}

\begin{proof}
This is a consequence of lemma~\ref{bruhatforloopgroups}.
\end{proof}

\begin{theorem}[Bruhat twin decomposition]
Let $\widehat{L}(G, \sigma)$ be an affine algebraic Kac-Moody group with affine Weyl group $W_{\textrm{aff}}$.
Let furthermore $B^\pm$ denote a positive and its opposite negative Borel group. There are two decompositions
 $$\widehat{L}(G,\sigma)=\coprod_{w\in W_{\textrm{aff}}}B^{\pm}wB^{\mp}\,.$$ 
\end{theorem} 

\begin{remark}
Note that the Bruhat twin decomposition is defined on the whole group $\widehat{L}(G,\sigma)$. For cities, this translates into the fact that any two chambers in $\mathfrak{B}^+$ resp.\ $\mathfrak{B}^-$ have a well-defined Weyl codistance --- see subsection~\ref{subsection:combinatorics_of_cities}. In contrast, Bruhat decompositions are only defined for the subgroups $\widehat{L}(G, \sigma)^{\pm}$. This translates into the fact that there are positive (resp.\ negative) chambers without a well-defined Weyl distance, hence that the buildings will be disconnected.
\end{remark}

\begin{example}
\label{Kumarsexample}
Shrawan Kumar studies formal completions of Kac-Moody groups and Kac-Moody algebras. Those groups complete in only  \glqq one direction\grqq\/ i.e.\ with respect to one of the two opposite $BN$\ndash pairs. Similarly in the setting of affine Kac-Moody groups of holomorphic loops we could use holomorphic functions with finite principal part. There is an associated twin $BN$\ndash pair; the positive Borel subgroups are completed affine Borel subgroups while the negative ones are the algebraic affine Borel subgroups. Thus for a geometric twin $BN$\ndash pair we have to use: $\widehat{L}(G,\sigma)^+=\widehat{L}(G,\sigma)$ and $\widehat{L}(G,\sigma)^-=\widehat{L_{alg}G}^{\sigma}$ (Kumar studies only the affine building associated to $BN^+$ --- cf.~\cite{Kumar02}).
\end{example}

\begin{definition}
An involution $\varphi:\widehat{L}(G, \sigma)\longrightarrow \widehat{L}(G, \sigma)$ is called a $BN$\ndash flip iff 
\begin{enumerate}
	\item  $\varphi^2=1$,
	\item $\varphi(B^+)=B^-$,
 	\item $\varphi$ centralizes $W$.
\end{enumerate}
\end{definition}

\begin{definition}
A geometric twin $BN$\ndash pair is called symmetric iff it has a $BN$\ndash flip. 
\end{definition}

\begin{example}
For a twin $BN$\ndash pair to be symmetric we need that $B^+$ and $B^-$ are isomorphic groups. More precisely this means that the completion has to be symmetric in both directions.
Thus the groups of example~\ref{Kumarsexample} have non symmetric geometric twin $BN$\ndash pairs. 
\end{example}

\begin{example}
An algebraic affine twin $BN$\ndash pair is symmetric.
\end{example}

\begin{example}
The geometric $BN$\ndash pair associated to any group $\widehat{L}(G, \sigma)$ is symmetric.
\end{example}

\begin{lemma}
The intersection $\widehat{L}(G,\sigma)^0$ of  $\widehat{L}(G,\sigma)^+$ with $\widehat{L}(G,\sigma)^-$ is a quasi-algebraic subgroup. For the standard affine Borel subgroups, it is the algebraic Kac-Moody group
$$\widehat{L}(G,\sigma)^0\simeq \widehat{L_{alg}G}^{\sigma}\,.$$
\end{lemma}

\begin{proof}
$\widehat{L_{alg}G}^{\sigma}$ is the maximal subgroup of $\widehat{L}(G,\sigma)$ having both Bruhat decompositions.
\end{proof}

\subsection{Combinatorics of cities}
\label{subsection:combinatorics_of_cities}

\noindent We now define a twin city using the $W$\ndash metric description:

\begin{definition}[Twin City]
Let $\widehat{L}(G,\sigma)$ be an affine Kac-Moody group with a geometric $BN$\ndash pair and Weyl group $W_{\textrm{aff}}$.
Define $\mathcal{C}^+:=\widehat{L}(G,\sigma)/B^+$ and $\mathcal{C}^-:=\widehat{L}(G,\sigma)/B^-$. 
\begin{enumerate}
\item The distance $\delta^{\epsilon}:\mathcal{C}^{\epsilon}\times \mathcal{C}^{\epsilon} \longrightarrow W_{\textrm{aff}}$, $\epsilon \in \{+,-\}$ is defined as usual via the Bruhat decompositions:
$\delta^{\epsilon}(gB^{\epsilon} , fB^{\epsilon})=w(g^{-1}f)$ if $g^{-1}f\in \widehat{L}(G,\sigma)^{\epsilon}$. Otherwise it is $\infty$.
\item The codistance  $\delta^*:\mathcal{C}^+\times \mathcal{C}^- \cup \mathcal{C}^-\times \mathcal{C}^+ \longrightarrow W$ is defined as usual by $\delta^*(gB^- , fB^+)=w^{\mp}(g^{-1}f)$ (resp.\ $\delta^*(gB^+ , fB^-)=w^{\pm}(g^{-1}f)$).
\end{enumerate}
\end{definition}

The elements of $\mathcal{C}^{\pm}$ are called the positive (resp.\ negative) chambers of the twin city. The building is denoted $\mathfrak{B}=\mathfrak{B}^+\cup\mathfrak{B}^-$. One can define a simplicial complex realization in the usual way. We define connected components in $\mathfrak{B}^{\pm}$ in the following way: Two elements $\{c_1, c_2\}\in \mathfrak{B}^{\pm}$ are in the same connected component iff $\delta^{\pm}(c_1, c_2)\in W_{\textrm{aff}}$. We will check that this is an equivalence relation. Denote the set of connected components by $\pi_0(\mathfrak{B})$ resp.\ $\pi_0(\mathfrak{B}^{\pm})$.

\begin{remark}
Let $\widehat{L}(G, \sigma)$ be an algebraic affine Kac-Moody group. Then each city consists of exactly one building: Hence the twin city coincides with the twin building.
\end{remark}

\begin{lemmas}[Properties of a twin city]~
\label{propertiesofgeometricuniversaltwinbuilding}
\begin{enumerate}
\item The connected components of $\mathfrak{B}^{\epsilon}$ are affine buildings of type $(W,S)$.
\item Each pair consisting of one affine building in $\mathfrak{B}^+$ and one in $\mathfrak{B}^-$ is a twin building of type $(W,S)$.
\item The connected components of $\mathfrak{B}^{\epsilon}$ are indexed by elements in $\widehat{L}(G, \sigma)/\widehat{L}(G,\sigma)^{\epsilon}$.
\end{enumerate}
\end{lemmas}

\begin{proof}~
\begin{enumerate}
\item Call two elements equivalent iff they are in the same connected component. This relation  is clearly symmetric and self-reflexive. To prove transitivity, let $fB^{\pm}$, $gB^{\pm}$ and $hB^{\pm}$ be such that there are $w_{fg}, w_{gh}\in W$ such that $f^{-1}g\in B^{\pm}w_{fg}B^{\pm}$ and $g^{-1}h\in B^{\pm}w_{gh}B^{\pm}$. Then $f^{-1}h=f^{-1}gg^{-1}h\in B^{\pm}w_{fg}B^{\pm}w_{gh}B^{\pm}$; hence the distance is in $W$ and thus finite. Connected components are exactly subsets with finite codistance. We have to check that each connected component fulfills the metric definition of a building. 
\begin{enumerate}
\item If $\delta(fB^{\epsilon}, gB^{\epsilon})=w$ --- hence there are $b_1, b_2\in B^{\epsilon}$ such that $f^{-1}g=b_1wb_2$ ---, then $g^{-1}f=b_2^{-1}w^{-1}b_1^{-1}$. Hence $\delta(gB^{\epsilon}, fB^{\epsilon})=w^{-1}$.
\item If $\delta(fB^{\epsilon}, gB^{\epsilon})=w$ and $\delta(f'B^{\epsilon},fB^{\epsilon})=s$, then 
$\delta(f'B^{\epsilon}, gB^{\epsilon})=w(f^{'-1}g)=w(f^{'-1}ff^{-1}g)\subset w(f^{'-1}f)w(f^{-1}g)\cup w(f^{-1}g)\in \{sw, w\}$. If $l(sw)=l(w)+1$, then $\delta(f'B^{\epsilon}, gB^{\epsilon})=sw$.
\item Let $\delta(fB^{\epsilon}, gB^{\epsilon})=w$ and denote by $C_s$ be the $s$\ndash panel containing $fB^{\epsilon}$. There are two possibilities: 
	\begin{enumerate}
	\item Either $w$ has a representation such that $w=sw'$ and  $l(sw')>l(w')$ --- i.e.\ the first letter of any reduced word representing $w'$ in the generators $s_i$ is not $s$. As the last letter of $w$ is $s$, the last chamber of the gallery connecting $fB^{\epsilon}$ and  $gB^{\epsilon}$, denoted $f'B^{\epsilon}$, is contained in the $s$\ndash panel $C_s$. Hence $\delta(f'B^{\epsilon}, gB^{\epsilon})=w'=sw$. 
	\item If $w$ has no representation of the form $w=sw'$ such that $l(sw')>l(w')$, then any chamber $fB^{\epsilon}$ in the panel $C(s)$ satisfies $\delta(f'B^{\epsilon}, gB^{\epsilon})=sw$. 
	\end{enumerate}
\end{enumerate}
\item Each pair consisting of one connected component in $\mathfrak{B}^+$ and one in $\mathfrak{B}^-$ fulfills the axioms of definition~\ref{algebraictwinbuildingwmetricapproach}. As the Bruhat decomposition is defined on $\widehat{L}(G, \sigma)$, the codistance is defined between arbitrary chambers in $\mathfrak{B}^{\epsilon}$ resp.\ $\mathfrak{B}^{-\epsilon}$.
\item $\widehat{L}(G,\sigma)$ has a decomposition into subsets of the form $\widehat{L}(G,\sigma)^{\epsilon}$. Those subsets are indexed with elements in $\widehat{L}(G,\sigma)/\widehat{L}(G,\sigma)^{\epsilon}$. The class corresponding to the neutral element is $\widehat{L}(G,\sigma)^{\epsilon}\subset \widehat{L}(G,\sigma)$. Thus it corresponds to a connected component and a building of type $(W,S)$.  The result follows via translation by elements in $\widehat{L}(G,\sigma)/\widehat{L}(G,\sigma)^{\epsilon}$: a connected component of $\mathfrak{B}^{\epsilon}$ containing $fB^{\epsilon}$ consists of all elements $ f\widehat{L}(G,\sigma)^{\epsilon} B^{\epsilon}$ as $\delta(fhB^{\epsilon}, fh'B^{\epsilon})=w((fh)^{-1}fh')= w(h^{-1}f^{-1}fh')=w(h^{-1}h')\in W$ as $h, h'\in \widehat{L}(G, \sigma)$. 
\end{enumerate}
\end{proof}

\begin{definition}
A twin city $\mathfrak{B}$ is symmetric iff there is a simplicial complex involution $\varphi_{\mathfrak{B}}: \mathfrak{B}\longrightarrow \mathfrak{B}$ such that $\varphi_{\mathfrak{B}}(\mathfrak{B}^{\epsilon})=\mathfrak{B}^{-\epsilon}$.
\end{definition}

\begin{lemma}
A twin city is symmetric iff its geometric $BN$\ndash pair is symmetric.
\end{lemma}

\begin{proof}
The $BN$\ndash pair involution induces a building involution.
\end{proof}

\subsection{Group actions on the twin city}

This section studies the action of $\widehat{L}(G, \sigma)$ on the twin city associated to it.

We recall the special case of an algebraic Kac-Moody group: Borel subgroups in an algebraic Kac-Moody group are exactly the stabilizers of chambers while parabolic subgroups are the stabilizers of simplices. Furthermore the action is isometric with respect to the Weyl distance.

\begin{lemmas}[Action of $\widehat{L}(G, \sigma)$]~
\begin{enumerate}
\item The action of $\widehat{L}(G, \sigma)$  on $\mathfrak{B}$ by left multiplication is isometric.
\item The Borel subgroups are exactly the stabilizers of the chambers, parabolic subgroups are the stabilizers of simplices.
\item $\widehat{L}(G, \sigma)^{\epsilon}$  acts on the identity component $\Delta^{\epsilon}_0\subset \mathfrak{B}^{\epsilon}$ by isometries.
\item Let $\Delta_1^+\cup \Delta_1^-$ be an arbitrary twin building in $\mathfrak{B}$. Suppose $D^{\pm}$ are two opposite Borel subgroup stabilizing cells in $\Delta^{\pm}_0$. The group  $G(D^+,D^-)$ acts on  $\Delta_0^+\cup \Delta_0^-$ by isometries.
\item Let $fB^{\epsilon}$ and $gB^{\epsilon}$ be two chambers in the same connected component of $\mathfrak{B}^{\epsilon}$, and $h, h'\in \widehat{L}(G, \sigma)^{-\epsilon}$. The left translates  $hfB^{-\epsilon}$ and $h'gB^{-\epsilon}$ are in the same connected component iff $f^{-1}h^{-1}h'g \in \widehat{L}(G,\sigma)^{\epsilon}$. 
\end{enumerate}
\end{lemmas}

\begin{proof}~
\begin{enumerate}
\item $G$ acts isometrically on a twin building if the action on both parts preserves the distances and the codistance --- cf.~\cite{AbramenkoBrown08}, section~6.3.1. Hence the first assertion follows from the definition of $\mathcal{C}^{\pm}$ as coset spaces of $\widehat{L}
(G, \sigma)$ and a direct check: $$\delta(hfB^{\pm\epsilon}, hgB^{\pm \epsilon})=w(f^{-1}h^{-1}hg)=w(f^{-1}g)=\delta(fB^{\pm\epsilon}, gB^{\pm \epsilon})\,.$$
\item The chamber corresponding to $fB^{\epsilon}$ is stabilized by the Borel subgroup $B_f^{\epsilon}:=fB^{\epsilon}f^{-1}$. The converse follows as each Borel subgroup is conjugate to a standard one. Analogous for the parabolic subgroups. 
\item The identity component is described by $\widehat{L}(G, \sigma)^{\epsilon}B^{\epsilon}$. Hence it is preserved by left multiplication of $\widehat{L}(G, \sigma)^{\epsilon}$ --- thus the action is well defined and isometric by the first statement. 

\item The group $\widehat{L}(G, \sigma)^{\epsilon}=B^{\epsilon}WB^{\epsilon}$ acts by the last statement on $\Delta_0^{\epsilon}$ by isometries. 
Hence $G(B^+,B^-)=\widehat{L}(G, \sigma)^+\cap\widehat{L}(G, \sigma)^-$ acts on $\Delta_0^+\cup \Delta_0^-$ by isometries.
Let $fB^{+}\in \Delta_1^{+}$ and $gB^{-}\in \Delta_1^{-}$ be the chambers stabilized by $D^{\pm}$. Then $D^{+}=fB^+f^{-1}$ and $D^-=gB^-g^{-1}$. 
By theorem~\ref{algebraic_subgroups} there is some $h\in \widehat{L}(G, \sigma)$ such that $G(D^+,D^-)=hG(B^+,B^-)h^{-1}$. The groups $B^{'\pm}=hB^{\pm}h^{-1}$ are Borel subgroups in $G(D^+,D^-)$. As all positive resp.\ negative Borel subgroups in $G(D^+, G^-)$ are conjugate in $G(D^+, G^-)$, there are elements $f'$ and $g'$ in $G(D^+, G^-)$ such that $D^+=f'{B'}^{+}{f'}^{-1} =f'hB^{+}h^{-1}{f'}^{-1}$ and $D^-=g'{B'}^{-}{g'}^{-1} =g'hB^{-}h^{-1}{g'}^{-1}$. Hence $f=f'hb_f$ with $b_f\in B^+$ and $g=g'hb_g$ with $b_g\in B^-$. 

Now we can prove that $\Delta_1^+$ is invariant under $G(D^+,D^-)$:
For $k\in G(D^+,D^-)$ we find using $k=hk_0h^{-1}$ (hence $k_0\in G(B^+,B^-)\subset\widehat{L}(G, \sigma)^+$):
\begin{equation*}
  \begin{aligned}
    \delta(kfB^+, fB^+)&=w(f^{-1}k^{-1}f)=\\
    &= w(f^{-1}hk_0^{-1}h^{-1}f)=\\
    &= w(b_fh^{-1}f'hk_0^{-1}f'hb_f)
  \end{aligned}
\end{equation*}

As $b_f\in B^+\subset \widehat{L}(G, \sigma)^+$, $h^{-1}{f'}^{-1}h\in G(B^+, B^-)\subset \widehat{L}(G, \sigma)^{+}$, $k_0\in G(B^+,B^-)\subset\widehat{L}(G, \sigma)^+$, we find that $w(b_fh^{-1}f'hk_0^{-1}f'hb_f)\in W$. Hence $G(D^+, D^-)$ preserves $\Delta_1^+$. 
Analogously we conclude for $\Delta_1^-$. This proves the claim.

\item $f^{-1}h^{-1}h'g \in \widehat{L}(G,\sigma)^{\epsilon}$ is equivalent to $\delta^{\epsilon}(hfB^{-\epsilon}, h'gB^{-\epsilon})\in W_{\textrm{aff}}$. \qedhere
\end{enumerate}
\end{proof}

As quasi-algebraic subgroups are in bijection with algebraic twin buildings in $\mathfrak{B}$, we give a geometric characterization of them.

\begin{theorem}
\label{characterization_of_quasi_algebraic_loops}
Let $H\subset \widehat{L}(G, Id)$ be a subgroup conjugate to $G$ and $\mathfrak{h}$ its Lie algebra. Then the group 
$$L_HG:=\{e^{tX_1} e^{t Y_1}\dots e^{tX_n} e^{tY_n} g \hspace{1pt}|\hspace{2pt} g\in H,\hspace{1pt} X_i, Y_i \in \mathfrak{h},\hspace{1pt} e^{tX_i} e^{t Y_i}=e  \}$$
is quasi-algebraic. Conversely for each quasi-algebraic subgroup $\widehat{L_{qalg}G}$ there are $H$ and $\mathfrak{h}$ such that $\widehat{L_{qalg}G}$ is of this form.
\end{theorem}

\noindent Before describing the proof, let us note as a corollary an application: We give a characterization of the connected components of $\mathfrak{B}^{\epsilon}$. 

\begin{cor}
\label{cor:implication_for_building}
Two simplices $\bar{f}B^{\epsilon}$ and  $\bar{g}B^{\epsilon} \in \mathfrak{B}^{\epsilon}$ are contained in the same connected component $\Delta_1^{\epsilon}$ of $\mathfrak{B}^{\epsilon}$ iff there are representatives $f,g\in \widehat{L}(G, \sigma)$, $\{X_1,\dots, X_n, Y_1, \dots, Y_n \in \mathfrak{h} |e^{tX_i} e^{t Y_i}=e\}$  where $\mathfrak{g}\simeq \mathfrak{h}\subset \widehat{L}(\mathfrak{g}, \sigma)$ and a constant $c$ such that $fB^{\epsilon}=\bar{f}B^{\epsilon}$ $gB^{\epsilon}=\bar{g}B^{\epsilon}$  and $f(t)= e^{tX_1} e^{t Y_1}\dots e^{tX_n} e^{tY_n} c \cdot g(t)$.
\end{cor}

\begin{proof}[Proof of corollary~\ref{cor:implication_for_building}]
Choose a quasi-algebraic subgroup $G(\Delta_1^{\epsilon})$ acting transitively on $\Delta_1^{\epsilon}$. Choose $g$ to be an arbitrary representative of $\bar{g}$. Then there is $f'\in G(\Delta_1)$ such that $f'g$ is a representative for $\bar{f}$. Put $f:=f'g$, choose an embedding $\varphi:\widehat{L_{alg}\mathfrak{g}}\longrightarrow G(\Delta_1^{\epsilon})$ and put $\mathfrak{h}=\varphi(\mathfrak{g})$.  Now the corollary follows from theorem~\ref{characterization_of_quasi_algebraic_loops}.
\end{proof}

\noindent The theorem is a consequence of the following lemma, due to Ernst Heintze, which is the special case for $H:=G$.

\begin{lemmas}[Characterization of $L_{alg}G$]~
\label{characterizationofalgebraicloops}
$$L_{alg}G:=\{e^{tX_1} e^{t Y_1}\dots e^{tX_n} e^{tY_n} g \hspace{1pt}|\hspace{2pt} g\in G,\hspace{1pt} X_i, Y_i \in \mathfrak{g},\hspace{1pt} e^{tX_i} e^{t Y_i}=e  \}$$
\end{lemmas}

\begin{proof}[Proof of theorem~\ref{characterization_of_quasi_algebraic_loops}]~
 \begin{enumerate}
   \item Choose $k\in L(G, \sigma)$ such that $H=kGk^{-1}$. Then $\mathfrak{h}=k\mathfrak{g}k^{-1}$. Hence $L_HG=kL_GGk^{-1}=kL_{alg}Gk^{-1}$.
   \item Let $L_{qalg}G$ be a quasi-algebraic subgroup of $L(G, \sigma)$. Then there is some $k\in L(G, \sigma)$ such that $L_{qalg}G=kL_{alg}Gk^{-1}$. Put $H=kGk^{-1}$ and $\mathfrak{h}=k\mathfrak{g}k^{-1}$ and we have reduced the statement of theorem~\ref{characterization_of_quasi_algebraic_loops} to the lemma~\ref{characterizationofalgebraicloops}.
 \end{enumerate}

\end{proof}

We now give the proof of lemma~\ref{characterizationofalgebraicloops}:

\begin{proof}[Proof of lemma~\ref{characterizationofalgebraicloops}] 
Define: $$L'_{alg}G:=\{e^{tX_1} e^{t Y_1}\dots e^{tX_n} e^{tY_n} g \hspace{1pt}|\hspace{2pt} g\in G,\hspace{1pt} X_i, Y_i \in \mathfrak{g},\hspace{1pt} e^{tX_i} e^{t Y_i}=e  \}$$

We have to show: $L'_{alg}G=L_{alg}G$.

\begin{itemize}
	\item[-] We show: $L'_{alg}G \subset L_{alg}G$.	
	First remark that $L'_{alg}G$ is a group of periodic mappings $c: \mathbb{R}\rightarrow G$ with period $1$. As $ge^{tX}=e^{tAd(g)X}g$, the product of two elements is again in $L'_{alg}G$. Checking the group axioms is then elementary. Thus $L'_{alg}G$ is a subgroup of $LG$. From theorem 4.7.\ in~\cite{Mitchell88} it follows that $c(t)=\exp{tX}\exp{tY}$ is in $L_{alg}G$ iff $\exp{tX}\exp{tY}=e$. As each element in $L'_{alg}G$ is generated by elements in $L_{alg}G$, we get: $L'_{alg}G \subset L_{alg}G$.
	\item[-] We show: $L_{alg}G\subset L'_{alg}G$.
	To prove this direction, we study the action of $L'_{alg}G$ on the building. We show:
	\begin{enumerate}
	  \item	$L'_{alg}G$ acts transitively on the set of chambers. 
	  \item The isotropy group of a chamber is the same for $L_{alg}G$ and $L'_{alg}G$.
	\end{enumerate}
Those two assertions contain the theorem, as for $g\in L_{alg}G$ we find the existence of a $g'\in L'_{alg}G$ such that $g\Delta_0=g' \Delta_0$ for some fixed chamber $\Delta_0$. Thus $g'^{-1}g\Delta_0=\Delta_0$. 
Thus the product $g'^{-1}g$ is in the isotropy group of $\Delta_0$ with respect to the $L_{alg}G$\ndash action, called $L_{alg}G_{\Delta_0}$.

Now the second assertion tells us: $g'^{-1}g\in L'_{alg}G_{\Delta_0}=L_{alg}G_{\Delta_0}$ Set $g'':=g'^{-1}g\in L'_{alg}G$. Then $g=g' g''\in L'_{alg}G$. Thus $L_{alg}G \subset L_{alg}G'$ and the lemma is proved. 
Thus we are left with checking assertions $1.$ and $2.$:
  \begin{itemize}
    \item[-] We prove:  The isotropy group of a chamber is the same for $L_{alg}G$ and $L'_{alg}G$.
  		Let $\mathfrak{B}_{alg}$ be the affine building, associated to $L_{alg}G$, $X\in \mathfrak{B} $ a cell of type $I$, $P_I$ its stabilizer in $LG_{\mathbb{C}}$. We know from \cite{Mitchell88}:
  		\begin{alignat*}{2}
  		L_{alg}G\cap P_I	&=\{\overline{h}\in L_{alg}G|h(t) \exp(tX)h^{-1}(1)=\exp(tX)\}=\\
  											&=\{\overline{h}\in L_{alg}G| h(t)=\exp(tX)h(1) \exp(-tX)\}\subset\\
  											&\subset L'_{alg}G
  		\end{alignat*}
  		The last inclusion is true, as $\exp(tX)h(1)\exp(-tX)=\exp(tX)\exp(tY)h(1)$ with $Y=-\textrm{Ad} h(1) X$ and $[h(1), \exp(tX)]=0$.
    \item[-] We prove: $L'_{alg}G$ acts transitively on the set of chambers.
          To this end, we remark that the action of $L_{alg}G\cap P_i\simeq SU(2)$ is transitive on the chambers having the panel corresponding to $i$ in its boundary. Transitivity on the building follows now as every pair of chambers can be connected by a gallery, which we can follow by repeated application of the transitivity on the chambers surrounding a panel.

 As those groups are in $L'_{alg}G$, the action of $L'_{alg}G$ is transitive on $\mathfrak{B}$ --- the result follows now. \qedhere
  \end{itemize}
\end{itemize}
\end{proof}

\begin{remark}
The strategy of proof is similar in spirit to the amalgam\ndash based local\ndash global constructions in the Kac-Moody theory.
\end{remark}

\section{Twin cities and Kac-Moody algebras}
\label{section:twin_cities_and_kac_moody_algebras}

In this section we describe an explicit realization of the twin city.

\noindent Define the two simplicial complexes:
$$
\begin{aligned}
\mathfrak{B}^+&=(L(G_{\mathbb C}, \sigma)/B^+\times \Delta)/\sim\,,\\
\mathfrak{B}^-&=(L(G_{\mathbb C}, \sigma)/B^-\times \Delta)/\sim\,. 
\end{aligned}
$$

In this description $B^+$ and $B^-$ denote opposite Borel subgroups, $\Delta$ denotes the fundamental alcove in a fixed torus $\mathfrak{t} \subset {\mathfrak{g}}$ and $\sim$ is the equivalence relation defined by $(f_1, Y_1)\sim (f_2, Y_2)$ iff $Y_1= Y_2\simeq Y$ and $f_1\simeq f_2 (\textrm{mod} \left(\textrm{Fix} \exp{tY}\right))$. Using the Iwasawa decomposition of $L(G_{\mathbb{C}}, \sigma)$ we get a second description:

$$
\begin{aligned}
\mathfrak{B}^+&=(L(G_{\mathbb R}, \sigma)/T\times \Delta)/\sim\,, \\
\mathfrak{B}^-&=(L(G_{\mathbb R}, \sigma)/T\times \Delta)/\sim\,. 
\end{aligned}
$$
Furthermore, we set 
$$\mathfrak{B}=\mathfrak{B}^+\cup \mathfrak{B}^-\,.$$

\begin{definition}[Apartment]
By abuse of notation, let $W_{\textrm{aff}}\subset L(G_{\mathbb{C}},\sigma)/B=L(G_{\mathbb{R}},\sigma)/T$ be a realization of the affine Weyl group of $G_{\mathbb{C}}$, $W^f_{\textrm{aff}}:=f W_{\textrm{aff}}f^{-1}$.
An apartment $\mathcal{A}_f^{\pm}\in \mathfrak{B}^{\pm}$ is the simplicial complex
$$\mathcal{A}_f^{\pm}:=(W^f_{\textrm{aff}}\times \Delta)/\sim\,.$$
\end{definition}

\begin{proof}~
\begin{itemize} 
	 \item[-] To check that the embedding $W_{\textrm{aff}}\subset G_{\mathbb{C}}/B=G_{\mathbb{R}}/T$ is well defined, let $\mathfrak{t}\subset \mathfrak{g}$ be a maximal Abelian subalgebra. Let $H:=\{g\in G|g\mathfrak{t}g^{-1}=\mathfrak{t}\}$. $H$ is a group. Let $X\in \mathfrak{t}$ be a regular element, $K:=\textrm{Fix}(X)\simeq T$. Then $W=H/T\subset G/T$.  
   \item[-] $\mathcal{A}_f^{\pm}$ is a thin Coxeter complex of type $W$. Thus $\mathcal{A}_f^{\pm}$ is an apartment. \qedhere 
   \end{itemize} 
  \end{proof}

\begin{lemma}
Two elements $(f, X), (g,Y) \in \mathfrak{B}^{\pm}$ are contained in the same connected component iff $f^{-1}g \in \widehat{L}(G, \sigma)^{\pm}$. 
\end{lemma}

\begin{proof}
This is a restatement of lemma~\ref{propertiesofgeometricuniversaltwinbuilding}.
\end{proof}

We want now to embed the  twin city in the compact real form $\widehat{L}(\mathfrak{g}_{\mathbb{R}}, \sigma)$ of a Kac-Moody Lie algebra. It will appear as a tessellation of a space $H_{l, r}$ defined as the intersection of the sphere of radius $l, l\in \mathbb{R}$, with a horosphere $r_d= \pm r$. The two sheets of this sphere will correspond to $\mathfrak{B}^+$ resp.\ $\mathfrak{B}^-$.

We require that the regularity of $\widehat{L}(G, \sigma)$  is such that the restriction of the gauge action of $L(G, \sigma)$ on $L(\mathfrak{g},\sigma)$ to $H_{l,r}$ is polar. This condition is fulfilled for $\widehat{A_nG}^{\sigma}$ and for $\widehat{MG}^{\sigma}$ as is shown in \cite{Freyn09}. For Kac-Moody groups of $H^1$\ndash loops acting on the Kac-Moody algebra of $H^0$\ndash loops it is a consequence of Terng's work --- cf.~\cite{Terng95}.

To construct the embedding, we start with the conjugation action:
$$\widehat{\varphi}: \widehat{L}(G,\sigma) \times \widehat{L}(G,\sigma) \longrightarrow \widehat{L}(G, \sigma), \hspace{15pt}(g,h) \mapsto ghg^{-1} $$

\noindent By differentiation we get the adjoint action on the Lie algebra:
$$\widehat{\varphi}: \widehat{L}(G, \sigma) \times \widehat{L}(\mathfrak{g}, \sigma) \longrightarrow \widehat{L}(\mathfrak{g},\sigma), \hspace{15pt}(g,\widehat{u}) \mapsto g\widehat{u}g^{-1}$$

In contrast to the finite dimensional theory it is not possible to cover $\widehat{L}(\mathfrak{g}, \sigma)$ with maximal conjugate flats.

\noindent In contrast, the polarity assumption shows this to be possible for the restriction to $H_{l,r}$ (which is invariant under the adjoint action). Hence we conclude that $H_{l,r}$ is covered with finite dimensional conjugate Abelian subalgebras.  So in the end the situation is exactly as in the finite dimensional case; hence the algebra works out exactly the same way:

\noindent We find for a finite dimensional flat $\widehat{\mathfrak{a}}$  
$$ 
\begin{matrix}
\widehat{\varphi}:& \widehat{L}(G, \sigma) \times H_{l,r} &\longrightarrow &H_{l,r}, &(g,\widehat{u}) \mapsto g\widehat{u}g^{-1}\\
\widehat{\varphi}:& \widehat{L}(G,\sigma) \times \widehat{\mathfrak{a}}\cap H_{l,r} &\longrightarrow& H_{l,r},& (g,\widehat{u}) \mapsto g\widehat{u}g^{-1}\,.
\end{matrix}
$$

Taking $\widehat{\mathfrak{a}}$ to be the standard flat (i.e.\ for non-twisted groups: $\mathfrak{a}$ consists of constant loops), we find that $\mathfrak{a}_H:= \widehat{\mathfrak{a}}\cap H$ consists of triples $\widehat{X}=(X, r_c, r_d)$ where $r_c$ is defined by the condition $|\widehat{X}|=l$.

The exponential image of $\widehat{\mathfrak{a}}$ is the Cartan subalgebra $\widehat{T}\simeq T \oplus \mathbb S^1\oplus\mathbb S^1$. As $\widehat{\mathfrak{a}}$ is fixed by $\widehat{T}$, we get a well defined surjective  action 
$$\begin{matrix}
\widehat{\varphi}:& \widehat{L}(G,\sigma)/\widehat{T} \times \mathfrak{a}_H &\longrightarrow &H,&(g,u_H) \mapsto gu_Hg^{-1}\end{matrix}\,. $$

\noindent The surjectivity of this map follows from the polarity of the adjoint action (see theorem~\ref{polaractiononmg}).

Using the equivalence $\widehat{L}(G, \sigma)/\widehat{T} \simeq L(G,\sigma)/T$ we get:
$$\begin{matrix}
\varphi:& L(G,\sigma)/T \times \mathfrak{a}_H &\longrightarrow &H,&(g,u_H) \mapsto gu_Hg^{-1}\end{matrix}\,. $$

Now the inner automorphisms of $\widehat{\mathfrak{a}}$ are the elements of the affine Weyl group $W_A:=N(T)/T$, so we may further restrict $\widehat{\mathfrak{a}}_H$ to a fundamental domain of the action of $W_A$, denoted $\Delta$. Then the map
$$\begin{matrix}\varphi:& L(G,\sigma)/T\times \Delta&\longrightarrow& H,& (gT,\widehat{u}_H) \mapsto g\widehat{u}_Hg^{-1}\end{matrix}$$ is again surjective.

We can now construct a chamber complex by identifying $\Delta$ with a simplex $\mathfrak{B}$ with boundary and taking its $\widehat{L}(G, \sigma)$\ndash translates.

This construction proves the following theorems: 

\begin{theorem}[Embedding of the twin city]
\label{Embedding of the twin city}
For each algebra $L(\mathfrak{g}, \sigma)$ there is a $2$\ndash parameter family of embeddings for the twin city, parametrised by $r$ and the norm $l$. Those embeddings are equivariant in the sense that:
\begin{displaymath}
\begin{xy}
  \xymatrix{
      \mathfrak{B}_G \ar[rr]^{\widehat{L}(G,\sigma)} \ar[d]_{\varphi_{lr}}   & &   \mathfrak{B}_G  \ar[d]^{\varphi_{lr}}  \\
      \widehat{L}(\mathfrak{g}_R,\sigma) \ar[rr]^{\textrm{Ad}(\widehat{L}(G, \sigma))}           &  &   \widehat{L}(\mathfrak{g}_R,\sigma)  
  }
\end{xy}
\end{displaymath}
\end{theorem}

We call this $2$\ndash parameter family the thickened twin city. 

This construction yields the following result:

\begin{cor}
\label{cor:tori=apartment_compact}
Suppose the adjoint action of $\widehat{L}(G, \sigma)$ induces a polar action. Every torus in a Kac-Moody group $\widehat{L}(G, \sigma)$ corresponds to a complete twin apartment of the thickened twin city.  
\end{cor}

More generally, we have:

\begin{theorem}
\label{tori=apartment}
There is a correspondence between twin apartments in the twin city and tori of finite type in $\widehat{L}(G_{\mathbb{C}}, \sigma)$.
\end{theorem}

\begin{remark}
Bertrand Rémy proves a similar result for arbitrary algebraic Kac-Moody groups, showing a correspondence between twin apartments and Cartan subalgebras --- cf.~\cite{Remy02}, section~10.4.3. . 
\end{remark}

\begin{remark}
 For Kac-Moody groups of the classical type, the result can be proven by linear representations of the twin cities as complexes of periodic flags in Hilbert spaces --- cf.~\cite{Freyn09} for the case $\widetilde{A}_n$ and a sketch for the other types and \cite{Freyn10b} for the details. The possibility of a similar construction for Kac-Moody groups of the exceptional types is an interesting open problem.
\end{remark}

\begin{proof}[Proof of theorem~\ref{tori=apartment}]
The embedding shows this theorem for twin apartments corresponding to tori in the Kac-Moody group $\widehat{MG}^{\sigma}$. 
We have to prove two directions:
\begin{itemize}
 \item[-] Let $A$ be an arbitrary twin apartment in the affine twin building $\Delta_+\cup\Delta_-\subset \mathfrak{B}^+\cup \mathfrak{B}^-$. Let $G(\Delta_+, \Delta_-)$ be the quasi-algebraic group associated to $\Delta_+\cup\Delta_-$. By Bertrand Rémy's result, $A$ corresponds to a torus in $G(\Delta_+, \Delta_-)$. The embedding of $G(\Delta_+, \Delta_-)$ as a subgroup in $\widehat{L}(G, \sigma)$ identifies $A$ with a torus in $\widehat{L}(G, \sigma)$.
\item[-] Let $T'$ be a torus in $\widehat{L}(G, \sigma)$.  As all tori of finite type are conjugate, there is a $g$ such that $T'=gTg^{-1}$, where $T$ is the standard torus. $T'$ corresponds to the apartment that is the translate by $g$ of the apartment corresponding to $T$.
\end{itemize}
\end{proof}

\begin{remark}
A second possible proof of theorem~\ref{tori=apartment} consists in identifying the stabilizers of twin apartments in $\mathfrak{B}$ and of tori in $\widehat{L}(G, \sigma)$. This is an adaption of the strategy used by Bertram Rémy in \cite{Remy02} to our setting.
A third proof is implicit in the description of apartment systems in section~\ref{section:the_spherical_building_at_infinity}.
A fourth proof constructs for each apartment a compact real form such that the apartment corresponds to a torus of this compact real form. Then the result follows from corollary~\ref{cor:tori=apartment_compact}.
\end{remark}

Let us remark that a further generalization is possible to Kac-Moody symmetric spaces:

\begin{theorem}
Every flat of finite type in a Kac-Moody symmetric space $\widehat{MG}^{\sigma}$ corresponds to a complete twin apartment of the thickened twin city. 
\end{theorem}

For an overview of the theory and the definition of Kac-Moody symmetric spaces see \cite{Freyn07}, for a detailed description see \cite{Freyn09}. Their adjoint actions correspond to $s$\ndash representations of involutions of affine Kac-Moody algebras. Christian Gross proved in \cite{Gross00} those to be polar for the action of $H^1$\ndash Kac-Moody groups on $H^0$\ndash Kac-Moody algebras. The result remains true in all other regularity conditions, i.e.\ smooth loops, $k$\ndash differentiable loops and holomorphic loops on  $A_n$ resp.\ $\mathbb{C}^*$.  The last case corresponds to the $s$\ndash representations of Kac-Moody symmetric spaces --- cf.~\cite{Freyn09}. The other settings --- especially the $s$\ndash representation studied in \cite{Gross00} --- are completions. Now the proof follows the blueprint of the Kac-Moody group case which we described.

Hence the infinite dimensional theory mirrors the finite dimensional theory.

\section{Topology and geometry of \boldmath$\mathfrak{B}$}
\label{section:topology_and_geometry}

There are three sundry ways to define a topology (resp.\ geometry) on the twin cities:
\begin{enumerate}
	\item A structure on the geometric realization of $\mathfrak{B}$.
	\item A structure on the set of chambers in $\mathfrak{B}$.
	\item A structure on the set of buildings in $\mathfrak{B}$.
\end{enumerate}

We will discuss these 3 ways in the following three subsections.

\begin{samepage}
\subsection{The structure on the geometric realization of \boldmath$\mathfrak{B}$}

The embedding of the twin city into spaces $H_{l,r}$ shows:

\begin{theorem}
 Let $\widehat{L}(G, \sigma)$ be a Hilbert\ndash, Banach\ndash\/ or Fréchet\ndash Lie group. The geometric realization of the positive (resp.\ negative) component $\mathfrak{B}^+$ (resp.\ $\mathfrak{B}^-$) of the twin city carries the same structure. 
\end{theorem}
\end{samepage}

Using the results about the analytic structure of the various Kac-Moody groups, as developed in \cite{PressleySegal86}, \cite{Terng95}, \cite{Popescu05}, and \cite{Freyn09}, we get the following corollary:

\begin{cors}[The most important examples]~
\label{bistamefrechet}\begin{enumerate}
\item Each city associated to $\widehat{MG}$ carries a natural tame Fréchet structure.
\item Each city associated to $\widehat{A_nG}$ carries a natural Banach space structure.
\item Each city associated to $\widehat{L^{\infty}G}$ carries a natural tame Fréchet structure.
\item Each city associated to $\widehat{L^{1}G}$ carries a natural Hilbert space structure.
\end{enumerate}
\end{cors}

Using the description of $\widehat{M\mathfrak{g}}$ as inverse limit of the algebras $\widehat{A_n\mathfrak{g}}$  --- cf.~ \cite{Omori97} and \cite{Freyn09}, we find this structure reflected in an inverse limit system $\{\mathfrak{B}_{\widehat{MG}}, \displaystyle\lim_{\longleftarrow}\mathfrak{B}_{A_nG}\}$. Thus the twin city for $\widehat{MG}$ is surrounded by a cloud of buildings corresponding to groups of weaker regularity.

\subsection{The structure on the set of chambers in \boldmath $\mathfrak{B}$}

As chambers in $\mathfrak{B}$ correspond bijectively to elements in the quotient $MG/T$, the space of chambers inherits the tame Fréchet topology of $MG/T$. Study the gauge action of $MG$ on $M\mathfrak{g}$. By theorem~\ref{polaractiononmg} it is a polar action. Let $X\in \mathfrak{t}$ be an element in the Lie algebra of $T$, such that $\{MG\cdot X\}$ is a principal orbit. As the stabilizer of $X$ is $T$, we have $\{MG\cdot X\}\simeq MG/T$.

Hence the space of chambers can be identified with an isoparametric submanifold. So the structure of the space of chambers is well understood. 

\begin{theorem}
 Let $S$ be an isoparametric PF\ndash submanifold of a Hilbert space. Suppose $S$ is homogeneous and it is the principal orbit of the gauge action of a Hilbert loop group $L^1G^{\sigma}$. Then the points in the isoparametric submanifold correspond bijectively to chambers in the associated city. Furthermore curvature spheres correspond to panels. 
\end{theorem}

For the definition of isoparametric submanifolds see ~\cite{Terng89} and \cite{PalaisTerng88}.
All known isoparametric submanifolds with higher codimension are of this type.

Conversely Ernst Heintze and Xiaobo Liu --- cf.~\cite{HeintzeLiu} --- prove the following theorem:
\begin{theorem}
 A complete, connected, full, irreducible isoparametric submanifold $M$ of an infinite dimensional Hilbert space $V$ with codimension $\neq 1$ is a principal orbit of a polar action.
\end{theorem}

The set $Q$ constructed in \cite{HeintzeLiu} has the structure of an affine algebraic building.
For more details cf.~\cite{Terng95}, \cite{HPTT}, \cite{Freyn10a}, and the references therein.
It is conjectured that all polar actions on Hilbert spaces correspond to $P(G,H)$ actions under suitable assumptions on the cohomogeneity. Very promising partial results in this direction due to Claudio Gorodski, Ernst Heintze and Kerstin Weinl exist. If this is the case one gets an equivalence between cities and isoparametric submanifolds of codimension $\not=1$ mirroring the situation described by Thorbergsson's theorem in the finite dimensional situation  --- cf.~\cite{Thorbergsson91}.

\subsection{The structure on the set of buildings in \boldmath$\mathfrak{B}$}

While the space of chambers and the simplicial realization allow a metric structure similar to the one of the subjacent Lie group, i.e.\ a Hilbert, Banach or Fréchet space structure, the situation is completely different for $\mathfrak{B}$ itself. The simple fact that the chambers belonging to a single building are dense in the space of all chambers shows that no refinement of a topology on the space of chambers will give a topology on the space of buildings. 

As we choose to define the twin city in terms of the geometric Kac-Moody groups, we will also describe the geometry and topology of those groups. We want two buildings in $\mathfrak{B}$ to be close iff there is a small group $\widehat{L}(G, \sigma)$, i.e.\ a group defined using strong regularity conditions, containing both of them. As the product of two functions of a given regularity is of the same regularity, we find that a distance defined in this way will be ultrametric. 

Hence we will show that a twin city carries a ultrametric pseudo distance.

For $\widehat{f}\in \widehat{L}(G, \sigma)$ let $f=\sum a_kz^k$ be the (matrix valued) associated Fourier series of the loop part. Recall that convergence conditions on the series $\sum |a_k|$ correspond to regularity conditions on $f$ --- cf.~\cite{GoodmanWallach84} for an extensive overview:

\begin{itemize}
 \item[-] $f$ is in $L^rG$ iff $\sum |a_k|k^r<\infty$.
 \item[-] $f$ is smooth iff $\sum |a_k|k^r<\infty$ for all $r\in \mathbb{N}$.
 \item[-] $f$ is holomorphic on $A_n$ iff $\sum |a_k|e^{kn}<\infty$.
 \item[-] $f$ is holomorphic on $\mathbb{C}^*$ iff $\sum |a_k|e^{kn}<\infty$ for all $n\in \mathbb{N}$.
\end{itemize}

At the moment of this writing, it is unclear if there is a suitable distance function which is meaningful in the whole range of regularity conditions. 

\paragraph{The \boldmath $\{A_n, \mathbb{C}^*\}$\protect\ndash setting}~\\\nopagebreak

To metrize the \glqq cloud\grqq\/ of buildings surrounding the tame twin city associated to a Kac-Moody symmetric space, we propose the following definition:

\begin{definition}
For $\Delta_0,\Delta_1 \in \mathfrak{B^\pm}$ and $x\in \Delta_0, y\in \Delta_1$ define 
$$\nu(x,y)= \textrm{max}_n \{\textrm{There is } f \in A_nG \textrm{ such that } f(x)= y\}$$ and $d(x,y)= e^{-\nu(x,y)}$. Then we put $d(\Delta_0, \Delta_1)=d(x,y)$. 
\end{definition}

This is equivalent to $$\nu(x,y)= \textrm{max}_n\{\textrm{There is a function } f \textrm{ such that } f(x)= y \textrm{ satisfying } \sum |a_k|e^{kn}<\infty\}.$$

\begin{lemma}[pseudo distance]
\label{pseudo distance}
$d$ is a ultrametric pseudo distance on the space of buildings in $\mathfrak{B}^{\epsilon}$.
\end{lemma}

\begin{proof}[Proof of lemma~\ref{pseudo distance}]~
\begin{enumerate}
 \item We prove that $d$ is a ultrametric pseudo distance on the space of chambers. To this end let $x, y, z\in \mathfrak{B}^{\epsilon}$ be chambers. We have to check:
	\begin{itemize}
	\item[-] symmetry: $f\in A^nG \Leftrightarrow f^{-1} \in A^nG$. Thus $\nu(x,y)=\nu(y,x)$ and $d(x,y)= d(y,x)$.
	\item[-] strong $\Delta$\ndash inequality: Let $d(x,y)= e^{-\nu(x,y)}$, $d(y,z)= e^{-\nu(y,z)}$. Thus there is a function $f_{xy}\in A^{\nu(x,y)}G$ such that $f(x)=y$ and a function $f_{yz}\in A^{\nu(y,z)}G$ such that $f(y)=z$. Without loss of generality suppose $\nu(x,y)\leq \nu(y,z)$. Thus $A^{\nu(x,y)}G\supset A^{\nu(y,z)}G$. So $f_{xz}=f_{xy} f_{yz}\in A^{\nu(x,y)}G$. Thus $d(x,z)= e^{-\nu(x,z)} \leq e^{-\nu(x,y)}=d(x,y)$. 
	\end{itemize}
 \item We have to check that the distance on the space of buildings is well defined. To this end let $x,x'\in \Delta_0$. There is a quasi-algebraic subgroup $G(\Delta_0)$ acting transitively on $\Delta_0$. Let $h\in G(\Delta_0)$ such that $x'=h(x)$. Clearly $d(x,x')=0$. The result follows now from the triangle inequality. \qedhere

\end{enumerate}
\end{proof}

\paragraph{\bf\boldmath The Hilbert space setting: $H^1$-loops acting on  $H^0$-spaces}

In many papers describing the geometry of Kac-Moody groups (see \cite{HPTT}, \cite{Terng89}, and \cite{Terng95}), the setting of $H^1$\ndash loops with values in a compact simple Lie group $G$, acting on the space of $H^0$\ndash loops in $\mathfrak{g}$, is used. Our results carry over to this setting: 

Nevertheless, describing $\mathfrak{B}^{\pm}= (LG\times \Delta)/\sim$, it seems meaningful to make some changes in the definition of the pseudo distance. As we defined it, the distance between two buildings depends on the convergence radius of the functions transforming one building into the other. For $H^1$\ndash functions this definition is useless: The space of buildings such that the distance is $0$ is just to big. So it seems meaningful to introduce another distance function:

\begin{definition}[$H^1$\ndash distance]
Let $\Delta_1,\Delta_2 \in \mathfrak{B}^{\epsilon}$. Let $fB^{\epsilon}\in \Delta_1$, $gB^{\epsilon}\in \Delta_2$ and let $fg^{-1}=\sum a_k e^{ikt}$ be the Fourier series expansion.   $\nu_r(fB^{\epsilon},gB^{\epsilon})=\textrm{max}_r \{ \sum k^r a_k< \infty\}$ and
$d_r(\Delta_1,\Delta_2)= e^{-\nu(fB^{\epsilon}, gB^{\epsilon})}$.
\end{definition}

\begin{lemma}
The $H^1$\ndash distance is a ultrametric pseudo distance.
\end{lemma}

\begin{proof}
The proof follows the pattern of the proof for lemma~\ref{pseudo distance}.
\end{proof}

\section{The spherical building at infinity}
\label{section:the_spherical_building_at_infinity}

An affine building $\mathfrak{B}_{\textrm{aff}}$ being a CAT$(0)$\ndash space, it has a natural boundary at infinity. This boundary has the structure of a spherical building $\mathfrak{B}^{\infty}$. Let $W_{\textrm{aff}}$ be the affine Weyl group of $\mathfrak{B}_{\textrm{aff}}$. Call a vertex $v$ special iff every reflection\ndash hyperplane is parallel to a hyperplane passing through $v$ and let $W$ denote the stabilizer of a special vertex. Then the affine Weyl group $W_{\textrm{aff}}$ can be described as a semidirect product of $W$ with an Abelian group $\mathbb{Z}^{k}$. $W$ is the Weyl group of $\mathfrak{B}^{\infty}$. Every affine apartment in $\mathfrak{B}_{\textrm{aff}}$ has a spherical apartment in $\mathfrak{B}^{\infty}$ as its boundary. The chambers of the spherical building $\mathfrak{B}^{\infty}$ correspond to equivalence classes of sectors in the affine building --- cf.~\cite{AbramenkoBrown08} and \cite{Ronan03}.

\subsection{Apartment systems in affine buildings}

Similar to the various possible completions of Kac-Moody groups, an affine building does not have a unique apartment system but admits a variety of different apartment systems with sundry properties. Let $\mathcal{A}$ be one of them. We call cells in $\mathfrak{B}^{\infty}$ inner cells with respect to $\mathcal{A}$ if they correspond to equivalence classes of sectors in apartments $A\in \mathcal{A}$. The set of inner cells is a simplicial complex which will be denoted $\mathfrak{B}^{\infty}_{\mathcal{A}}$. It is clearly a subcomplex of $\mathfrak{B}^{\infty}$.
In general this complex is not a building: While every apartment is a spherical Coxeter complex and for each pair of apartments $A_1$ and $A_2$ there is a chamber complex isomorphism $\phi:A_1\longrightarrow A_2$ fixing $A_1\cap A_2$ (hence all apartments are of the same type), it is not clear that each pair of chambers has to be in a joint apartment. Call those complexes prebuildings as they can be made into buildings by the definition of additional apartments --- compare the notion of hovel introduced in \cite{GaussentRousseau08}, which describes a similar situation in another setting. We quote theorem 11.89 of \cite{AbramenkoBrown08}:

\begin{theorem}
$\mathfrak{B}^{\infty}$ is a building with $\mathcal{A}^{\infty}$ as system of apartments iff $\mathcal{A}$ has the following property: For any two $\mathcal{A}$\ndash sectors, there is an apartment in $\mathcal{A}$ containing a subsector of each of them. 
\end{theorem}

\begin{proof}
cf.~\cite{AbramenkoBrown08}. 
\end{proof}

Hence $\mathfrak{B}^{\infty}_{\mathcal{A}}$ is a building only for comparatively few distinguished apartment systems. As our buildings and cities carry actions of a Kac-Moody group $\widehat{L}(G, \sigma)$ our principal interest is in apartment systems described via a group action.
Let $A$ be an apartment, $H$ a group acting on $\mathfrak{B}$ and $\mathcal{A}=\{hA|h\in H\}$. Let $C\subset A$ be a chamber and $\mathfrak{C}\subset A$ be a sector, let $B_C$ be the stabilizer of $C$ and $B_{\mathfrak{C}}$ be the stabilizer of $\mathfrak{C}_{\infty}$.

\begin{theorem}
 $\mathcal{A}^{\infty}$ is a building iff $G=B_{\mathfrak{C}}NB_{\mathfrak{C}}$, hence if $(B_{\mathfrak{C}}, N)$ is a $BN$\ndash pair in $G$. 
\end{theorem}

\begin{proof}
cf.~\cite{AbramenkoBrown08}, theorem~11.100.
\end{proof}

In an affine building there is a minimal apartment system, which corresponds to the apartments described by the action of the group of algebraic loops (i.e.\ the group $G(k[t,t^{-1}])$); the corresponding building at infinity is the one associated to the fraction field $k(t)$. 

There are many more apartment systems, many of which correspond to various completions of this group. In contrast, a spherical building --- and hence the building at infinity --- has a unique well-defined apartment system (see corollary~\ref{cor:apartmentsystem_unique}). Hence there is a variety of different buildings at infinity, partially ordered by inclusion, that are all subbuildings of a maximal building, namely $\mathfrak{B}^{\infty}$.

\subsection{Candidates for the spherical building at infinity}

Spherical buildings of rank at least $3$ roughly correspond to algebraic groups over fields --- cf.~\cite{AbramenkoBrown08}, chapter~9. A spherical building at infinity corresponds therefore to a spherical building over some function field.  In the holomorphic setting of $\widehat{MG}$, there are three main candidates for this function field:
\begin{itemize}
   \item[-] The field of rational functions on $\widehat{\mathbb{C}}$, $\mathbb{C}(z)= \mathbb{C}\left(\frac{1}{z}\right)$.
   \item[-] The formal completion of the field of rational functions $\mathbb{C}((z))$ (resp.\ $\mathbb{C}\left(\left(\frac{1}{z}\right)\right)$).
   \item[-] The fields of meromorphic functions $\mathcal{M}(\mathbb{C}^*)$.
\end{itemize}

As $\mathbb{C}(z)= \mathbb{C}\left(\frac{1}{z}\right)$, the involution $z\mapsto \frac{1}{z}$ fixes $\mathbb{C}(z)$. Hence in this situation the spherical building associated to $\mathbb{C}(z)$ can be interpreted as the spherical building at infinity for affine buildings in $\mathfrak{B}^+$ and $\mathfrak{B}^-$.  

This is no longer true for $\mathbb{C}((z))$ and $\mathbb{C}\left(\left(\frac{1}{z}\right)\right)$ as these fields are completions of $\mathbb{C}(z)$ with respect to different norms.
Hence in this case we have to clearly distinguish between the positive and the negative buildings.

The spherical buildings associated to the fields $\mathbb{C}((z))$ (resp.\ $\mathbb{C}\left(\left(\frac{1}{z}\right)\right)$) correspond to the CAT$(0)$\ndash boundary $\mathfrak{B}^{\infty}$ of a building in $\mathfrak{B}^+$ (resp.\ $\mathfrak{B}^-$). The map $z\mapsto \frac{1}{z}$ swaps the two fields. Hence in this case one has to distinguish between positive and negative buildings.

The field of meromorphic functions, the fraction field of the field of holomorphic functions, is a special completion of $\mathbb{C}(z)$ in both directions. We have valuations for each interior point $z\in \mathbb{C}^*$. The affine buildings we are interested in correspond to $0$ and $\infty$, the only two points where no valuation is available, as functions in $\mathcal{M}(\mathbb{C}^*)$ can have essential singularities in these points. Hence the construction of a valuation  --- a necessity for the spherical building to be the boundary of an affine building --- fails. On the other hand, we get valuations on subrings. This is the correct situation for the building to be the boundary of a (twin) city. Let us remark, that one can use this observation as a starting point for a more abstract approach to twin cities.

\subsection{Spherical buildings for twin buildings}

Take an affine twin building $\Delta$. Both affine buildings $\Delta^+$ and $\Delta^-$ each have a spherical building at infinity. The twinning of the affine twin building induces a twinning of the spherical buildings at infinity. The easiest way to see this is to use the description of a twinning via twin apartments.  In the next theorem we call cells in the spherical buildings at infinity twin cells if they are interior cells with respect to a system of twin apartments. They are $\mathcal{A}$\ndash cells  for the apartment system described by the quasi-algebraic subgroup corresponding to the twin building. We quote the following theorem of Marc Ronan~\cite{Ronan03}:

\begin{theorem}
Let $\Delta=\Delta^+\cup \Delta^-$ denote an affine twin building and let $\Delta^{\infty,\pm}_{\mathcal{A}}\subset \Delta^{\infty,\pm}$ denote the subcomplex comprising all twin cells.  Then $\Delta^{\infty,+}_{\mathcal{A}}$ and $\Delta^{\infty,-}_{\mathcal{A}}$ are spherical subbuildings of $(\Delta_\pm)^{\infty}$, and the twinning of $\Delta^+_{\textrm{aff}}$ and $\Delta^-_{\textrm{aff}}$ defines a canonical isomorphism between $\Delta^{\infty,+}_{\mathcal{A}}$ and $\Delta^{\infty,-}_{\mathcal{A}}$. The complexes $\Delta^{\infty,\pm}_{\mathcal{A}}$  are spherical buildings over the field $\mathbb{C}(z)$.
\end{theorem} 

\begin{proof}
 This is the main theorem in \cite{Ronan03}.
\end{proof}

Hence twin apartments induce not only a twinning of the affine buildings but also of the spherical buildings at infinity.

Generalizing this idea to a twin city, each positive spherical building is twinned with each negative one along a series of inner apartments with respect to the apartment system corresponding to the quasi-algebraic subgroup $\widehat{G}(B^+, B^-)$ described by the twin building. This gives us a collection of spherical buildings at infinity which are in bijection with the individual twin buildings in $\Delta$.

\subsection{Embedding of the building at infinity}

Let $\mathfrak{B}$ be a building in the $X$\ndash category. Buildings in $\mathfrak{B}^{\pm}$ piece together to give an $X$\ndash vector space. Their spherical buildings at infinity piece together to form the sphere at infinity of this vector space. We can identify this sphere with a small sphere around a special vertex.  Chambers in the building at infinity correspond to equivalence classes of sectors. Using this identification they can be identified with the sectors based at a single chosen special vertex.

We verified that each affine building in $\mathfrak{B}^+$ is twinned with all buildings in $\mathfrak{B}^-$. Identify the cells of  $\mathfrak{B}^+$ as usual with $(G, \sigma)/B^+$ and take as system of apartments $\mathcal{A}_{L(G, \sigma)}$ all apartments corresponding to translates by elements of $L(G, \sigma)$ of a standard apartment.

We study now the structure at infinity described by this apartment system.

We investigate now how the $2$\ndash parameter family of immersion of $\mathfrak{B}$ behaves with respect to the various apartment systems. The crucial point is that the embedding \glqq looses\grqq\/ many apartments. Because of the identification $\widehat{L}(G_{\mathbb{C}}, \sigma)/B=L(G,\sigma)/T$ we just see a very limited number of apartments, namely each cell is contained in exactly one apartment. This means that we see in the embedding only one distinguished element in each equivalence class of sectors. Hence for each cell in $\mathfrak{B}^{\infty}_{\mathcal{A}}$ we are left with exactly one apartment containing it.

Let $\widehat{T}=T\oplus \mathbb{R}c\oplus \mathbb{R}d$ be a maximal torus and study the restriction of the 2\ndash parameter family $\varphi_{l,r}|_{\varphi^{-1}(\widehat{T})}$ of the embeddings of the twin city to $\widehat{T}$. Take a special vertex $v_0$ (that is, a vertex which contains an element in every parallel class). This vertex corresponds to a $2$\ndash subspace described by $l$ and $r$ passing through the origin. The sectors centered at $v_0$ are $2$\ndash parameter families of sectors in $\widehat{T}\backslash \widetilde{T}$. Take the closure of a sector in $\widetilde{T}$. This is again a $2$\ndash parameter family in $\widetilde{T}$. This yields the following theorem:

\begin{theorem}[Embedding of the spherical pre-building at infinity]
The cone over $\mathfrak{B}^{\infty}$ embeds equivariantly and surjectively in the loop algebra $L(\mathfrak{g},\sigma)$. Equivalently: there is a $2$\ndash parameter family of embeddings of the spherical buildings at infinity in $\widetilde{L}(G, \sigma)$. 
\end{theorem}

Via this embedding we get an identification of the spherical buildings associated to $\mathfrak{B}^+$ resp.\ $\mathfrak{B}^-$.

We can identify this loop algebra with the algebra $\{\widehat{u}\in \widehat{L}(\mathfrak{g}, \sigma)|r_d=0\}$.

\begin{proof}
The proof follows directly the pattern of the proof of theorem~\ref{Embedding of the twin city}.
\end{proof}

\noindent Continuity of the Adjoint action yields the result:

\begin{theorem}
The embedding of the cone building (i.e.\ the cone over the building, the additional vertex is mapped onto 0) is the closure of the embedding of the thickened twin city. The Adjoint action of $\widehat{L}(G,\sigma)$ on $\widehat{L}(\mathfrak{g},\sigma)$ preserves the identification of sectors in the twin city with chambers in the building $\mathfrak{B}^{\infty}$.
\end{theorem}

\begin{remark}
Caution: Sectors in the spherical complex at infinity do NOT correspond to infinite dimensional tori in $L\mathfrak{g}$! We constructed them as intersections of tori of finite type with the space $r_d=0$.
\end{remark}

\subsection{The building at infinity for \boldmath$\mathfrak{B}$}

It is a curious fact, that the representation theory of the loop algebra behaves differently to the one of the Kac-Moody algebra --- cf.~\cite{Khesin09}, remark~1.20.   The main point is that the codimension for orbits of the adjoint action is finite for the Kac-Moody algebra and infinite  for the loop algebra.
This geometric fact has a building theoretic interpretation: The adjoint action tells us, that the codimension should correspond to the dimension of maximal flats, hence to tori, hence to apartments in the associated building. If the $d$\ndash coefficient $r_d\not=0$, then the adjoint action of the Kac-Moody algebra describes the affine building over the finite dimensional fields $\mathbb{R}$ resp.\ $\mathbb{C}$. Hence the torus is finite dimensional. On the other hand, the loop group corresponds to the realization of an affine algebraic group scheme over a ring $R$ of functions, infinite dimensional over the ground field $\mathbb{R}$ resp.\ $\mathbb{C}$. Hence, maximal Abelian subgroups are finite dimensional over $R$, but infinite dimensional over the ground field. The associated algebraic group over the quotient field of $Q(R)$ has tori of rank $\textrm{rank}(G)$ over $Q(R)$. This algebraic group corresponds to the structure at infinity of the affine building associated to $R$, which is a spherical building over the function field $Q(R)$, hence finite dimensional over the function field but infinite dimensional over $\mathbb{R}$ or $\mathbb{C}$.

\begin{theorem}
\label{theorem:spherical_building_for_mg}
Let $\mathfrak{B}$ be a twin city in the category of holomorphic functions on $X\in \{A_n, \mathbb{C}^*\}$ with affine Weyl group $W_{\textrm{aff}}=W\ltimes \mathbb{Z}^k$. $\mathfrak{B}^{\infty}$ is the spherical building of type $W$ over the field $\mathcal{M}(X)$ of meromorphic functions of $X$.
\end{theorem}

\begin{remark}
The proofs are essentially an adaption of the known proofs to the case of the function field of meromorphic functions. The main difference is that the valuation we used to define the affine buildings is no longer defined on $\mathcal{M}(X)$ but only on subrings. This leads to the resulting affine building being disconnected, i.e.\ a twin city.
\end{remark}

\begin{remark}
The meromorphic functions on a subset $X\subseteq \mathbb{C}$ form a field. Let $W$ be a spherical Weyl group  $W\in \{A_n, B_{n, n\geq 2}, C_{n, n\geq 3}, D_{n, n \geq 4}, E_6, E_7, E_8, F_4, G_2\}$. Then the building of type $W$ over $\mathcal{M}(X)$ is well-defined. For the classical groups, this building can be described using the usual flag complexes.
\end{remark}

\begin{example}[{\bf SL(n)}]
The affine building associated to $\mathbb{C}[z,z^{-1}]$ is the complex of periodic flags in the $\mathbb{C}$\ndash vector space $V_{pol}$ of polynomial maps $f:\mathbb{C}^*\longrightarrow \mathbb{C}^n$ --- cf.~\cite{Garrett97}, \cite{AbramenkoRonan98}, \cite{Freyn09}, and numerous other sources. Interpret $V_{pol}$ as a module over the ring $\mathbb{C}[z, z^{-1}]$. We get a $n$\ndash dimensional vector space $V_{pol}^n:=V_{pol}\otimes_{\mathbb{C}[z, z^{-1}]} \mathbb{C}(z)$. The flag complex of this vector space is the spherical building at infinity, yielding an affine twin building $\Delta$ together with its spherical building at infinity $\Delta^{\infty}$.
The same construction works for the ring $\textrm{Hol}(\mathbb{C}^*)$ and its quotient field $\mathcal{M}(\mathbb{C}^*)$, yielding a twin city $\mathfrak{B}$ and its spherical building at infinity $\mathfrak{B}^{\infty}$ as the flag complex of a vector space $V^n_{hol}$.
Every embedding $\varphi: V^n_{pol}\longrightarrow V^n_{Hol}$ defines an embedding of the associated spherical buildings. Let $\varphi_0$ denote the canonical embedding and let $f\in MG$. Then $\varphi_f:=f\circ \varphi$ defines another embedding. While the group $SL(n, \mathbb{C}(z))\subset SL(n, \mathcal{M}(\mathbb{C}^*))$ acts on the canonically embedded building, the conjugate group $f SL(n, \mathbb{C}(z))f^{-1}$ acts on the building embedded via $\varphi_f$. This is the spherical building at infinity corresponding to the affine building corresponding to the quasi-algebraic subgroup $fSL(\mathbb{C}[z,z^{-1}])f^{-1}$ in the twin city. Details will appear in \cite{Freyn10b}.
 \end{example}

\begin{remark}
Suppose $\sigma=Id$ and focus only on the affine group scheme $G$. Homomorphisms of $\mathbb{C}(z)$ into $\mathcal{M}(\mathbb{C}^*)$ correspond to elements  $f\in \mathcal{M}(X)$ by mapping $z$ to $f(z)$. The field of rational functions in the \glqq variable\grqq\/ $f(z)$, denoted $\mathbb{C}(f(z))$ is a subfield of the field $\mathcal{M}(X)$. Hence the groups $G(\mathbb{C}(f(z)))$ embed into the group $G(\mathcal{M}(z))$. Each field homomorphism defines a group homomorphism. Thus the building over $\mathcal{M}(X)$ contains subbuildings for each subgroup  $G(\mathbb{C}(f(z)))$, which are buildings of the same type over the group field $\mathbb{C}(f(z))\simeq \mathbb{C}(z)$. Nevertheless, the torus extension, corresponding to the $c$ and $d$\ndash part of the Kac-Moody algebra makes the construction much more rigid. Let $w$ be in $\mathbb{C}^*$ factor in the Kac-Moody group corresponding to $d$. The action of $w$ on a loop $f(z)$ is defined by $w\cdot f(z)=f(zw)$. As $w\cdot z=wz$  equivariance of the homomorphism gives us: $w f(z)=f(zw)$. Hence $f(z)$ is a linear function, i.e.\ $f(z)=kz$, $k\in \mathbb{C}$. This shows, that on the level of fields, an embedding is sharply restricted by the extension structure. 
\end{remark}

\begin{proof}[Proof of theorem~\ref{theorem:spherical_building_for_mg}]
Let $W$ be the spherical Weyl group of $G(\mathcal{M}(\mathbb{C}^*))$ and $G(\mathbb{C}(z))$. Without loss of generality let $B_{\mathcal{M}(\mathbb{C}^*)}$ and $B_{\mathbb{C}(z)}$ denote the standard Borel subgroups (i.e.\ the upper\ndash triangular matrices) of the two groups. Then 
$G(\mathcal{M}(\mathbb{C}^*))=B_{\mathcal{M}(\mathbb{C}^*)}WB_{\mathcal{M}(\mathbb{C}^*)}$ and $G(\mathbb{C}(z))=B_{\mathbb{C}(z)}WB_{\mathbb{C}(z)}$. Let $\varphi_0:G(\mathbb{C}(z))\longrightarrow G(\mathcal{M}(\mathbb{C}^*))$ be the canonical embedding, (i.e.\ the embedding induced by the embedding of $\mathbb{C}(z)$ into $\mathcal{M}(\mathbb{C}^*)$ as the subfield of rational functions). Then $\varphi_0\left(B_{\mathbb{C}(z)}\right)=B_{\mathcal{M}(\mathbb{C}^*)}\cap \varphi_0 \left(G(\mathbb{C}(z))\right)$.
\begin{enumerate} 
\item Define $\varphi_f=\textrm{Ad}(f)\circ \varphi_0$. The following diagram commutes:
\begin{displaymath}
\begin{xy}
  \xymatrix{
      G(\mathbb{C}[z,z^{-1}]) \ar[rr]^{\otimes Q(\mathbb{C}[z,z^{-1}])} \ar[d]_{\varphi_f}   & &   G(\mathbb{C}(z))  \ar[d]^{\varphi_f}  \\
      MG \ar[rr]^{\otimes Q(\textrm{Hol}(\mathbb{C}^*))}           &  &   G(\mathcal{M}(\mathbb{C}^*))  
  }
\end{xy}
\end{displaymath}

Here $\otimes Q(R)$ denotes the tensor product with the fraction field $Q(R)$ of an integral domain $R$.

\item As $B_{\mathbb{C}(z)}=B_{\mathcal{M}(\mathbb{C}^*)}\cap G(\mathbb{C}(z))$, the orbit of $f B_{\mathcal{M}(\mathbb{C}^*})$ in $G(\mathcal{M}(\mathbb{C}^*))/B_{\mathcal{M}(\mathbb{C}^*)}$ under the action of $\varphi_f( G(\mathbb{C}^*))\subset G(\mathcal{M}(\mathbb{C}^*))$ is isomorphic to $G(\mathbb{C}(z))/B_{\mathbb{C}^*}$.  Hence $\varphi_f$ induces an embedding of the spherical building over the field $\mathbb{C}(z)$ into the spherical building over the field $\mathcal{M}(\mathbb{C}^*)$. The group $\varphi_f(G(\mathbb{C}(z)))$ operates transitively on this building.

\end{enumerate}

\end{proof}

\subsection{\boldmath The Hilbert space setting: $H^1$-loops acting on  $H^0$-spaces}

The spherical building at infinity for a twin city over a ring of functions $R$ is the spherical building over the fraction field $Q(R)$. The prove follows the lines for the one given in the last section --- for details about the situation of Hilbert-Lie groups --- cf.~\cite{Freyn10b}.

\bibliographystyle{alpha}
\bibliography{Doktorarbeit}

\end{document}